\documentclass[11pt,a4paper]{article}
\date{}

\usepackage{amssymb}
\usepackage{latexsym}
\usepackage{amsmath,amssymb}
\usepackage{amsthm,enumerate,verbatim}
\usepackage{amsfonts}
\usepackage{graphicx}
\usepackage{xcolor}
\usepackage{subfigure}
\usepackage{multimedia}
\usepackage[lined,boxed,commentsnumbered, ruled]{algorithm2e}
\numberwithin{equation}{section}
\newtheorem{assumption}{Assumption}
\newtheorem{definition}{Definition}[section]
\newtheorem{theorem}{Theorem}[section]
\newtheorem{lemma}{Lemma}[section]
\newtheorem{example}{Example}[section]

\newtheorem{remark}{Remark} [section]

\newcommand{\R}{{\mathbb{R}}}
\newcommand{\IN}{{\mathbb{N}}}
\newcommand{\dist}{{\rm dist}}

\title{Incremental Subgradient Methods for Minimizing The Sum of Quasi-convex Functions}

\author{Yaohua Hu\thanks{College of Mathematics and Statistics, Shenzhen University, Shenzhen 518060, P. R. China (mayhhu@szu.edu.cn). This author's work was supported in part by the National Natural Science Foundation of China (11601343) and Natural Science Foundation of Guangdong (2016A030310038).}
\and Carisa Kwok Wai Yu\thanks{Department of Mathematics and Statistics, Hang Seng Management College, Hong Kong (carisayu@hsmc.edu.hk). This author's work was supported in part by grants from the Research Grants Council of the Hong Kong Special Administrative Region, China (UGC/FDS14/P03/14 and UGC/FDS14/P02/15).}
\and Xiaoqi Yang\thanks{Department of Applied Mathematics, The Hong Kong Polytechnic University, Kowloon, Hong Kong (mayangxq@polyu.edu.hk). This author's work was supported in part by the Research Grants Council of Hong Kong (PolyU 152167/15E).}
}

\begin{document}
\maketitle

%
%


\noindent {\bf Abstract}\quad
The sum of ratios problem has a variety of important applications in economics and management science, but it is difficult to globally solve this problem. In this paper, we consider the minimization problem of a sum of a number of nondifferentiable quasi-convex component functions over a closed and convex set, which includes the sum of ratios problem as a special case. The sum of quasi-convex component functions is not necessarily to be quasi-convex, and so, this study goes beyond quasi-convex optimization. Exploiting the structure of the sum-minimization problem, we propose a new incremental subgradient method for this problem and investigate its convergence properties to a global optimal solution when using the constant, diminishing or dynamic stepsize rules and under a homogeneous assumption and the H\"{o}lder condition of order $p$. To economize on the computation cost of subgradients of a large number of component functions, we further propose a randomized incremental subgradient method, in which only one component function is randomly selected to construct the subgradient direction at each iteration. The convergence properties are obtained in terms of function values and distances of iterates from the optimal solution set with probability 1. The proposed incremental subgradient methods are applied to solve the sum of ratios problem, as well as the multiple Cobb-Douglas productions efficiency problem, and the numerical results show that the proposed methods are efficient for solving the large sum of ratios problem.


\section{Introduction}
In recent years, a great amount of attention has been attracted to the research of minimizing a sum of a number of nondifferentiable component functions:
\begin{equation}\label{eq-SP}
  \begin{array}{ll}
  \text{min}&f(x):=\sum_{i=1}^m f_i(x)\\
     \text{s.t.}  & x\in X,
    \end{array}
\end{equation}
where $f_i:\R^n\to \R$, $i=1,\dots,m$, are real-valued functions, and $X\subseteq \R^n$ is a closed set.

The type of convex sum-minimization problems, i.e., problem \eqref{eq-SP} with each $f_i$ being convex and $X$ being convex, has been widely studied in various applications, such as the Lagrangian dual of the coupling constraints of large-scale separable optimization problem \cite{Bertsekas11,Nedic01}, the distributed optimization problem in large-scale sensor networks \cite{BertsekasTsitsiklis96,RaN05} and the empirical risk minimization problem in online machine learning \cite{DuchiJMLR11,Mairal2015}. Motivated by vast applications of problem \eqref{eq-SP}, the development of optimization algorithms has become an important issue of the sum-minimization problem, and many practical numerical algorithms have been proposed to solve problem \eqref{eq-SP}; see \cite{Bertsekas99,Bertsekas11,HuSIOPT16,Mairal2015,XiaoTZhang14} and references therein.
In particular, the class of subgradient methods are popular and effective iterative algorithms for solving the large-scale convex sum-minimization problem \eqref{eq-SP}, due to the  simple formulation and low storage requirement. The subgradient method was originally introduced to solve a nondifferentiable convex optimization problem by Polyak \cite{Polyak67} and Ermoliev \cite{Ermoliev66}, and until now, various variants of subgradient methods have been studied to solve structured optimization problems; see \cite{Auslender09,Bertsekas2015,Kiwiel04,LarssonPS96,Nedic01,Nesterov09,Shor85} and references therein.
To meet the structure of sum-minimization problem \eqref{eq-SP}, an idea of incremental approach has been proposed to perform the subgradient process incrementally, by sequentially taking steps along the subgradients of component functions, with intermediate adjustment of the variables after processing each component function. That is, an iteration of the incremental subgradient method can be viewed as a cycle of $m$ subiterations, starting from $z_{k,0}:=x_k$, through $m$ steps
\begin{equation}\label{eq-incre}
z_{k,i}:=P_X \big(z_{k,i-1}-v_k g_{k,i}\big),\quad g_{i,k}\in \partial f_i(z_{k,i-1}),~~i=1,\ldots,m,
\end{equation}
and finally arriving at $x_{k+1}:=z_{k,m}.$
The incremental subgradient method has gained successful applications in large-scale sensor networks and online machine learning; see, e.g., \cite{DuchiJMLR11,RaN05}. So far, many articles have been devoted to the study of convergence analysis and applications of different types of incremental subgradient methods; see \cite{Bertsekas99,JohanRJ09,Kiwiel04,Nedic01,NetoPierro09,RamNV09,Shi09} and references therein.
In particular, Nedi\'{c} and Bertsekas \cite{Nedic01} investigated the convergence properties of the incremental subgradient methods, including the deterministic and stochastic ones, for the constant/diminishing/dynamic stepsize rules; later the authors extended these convergence results to the inexact incremental subgradient method with the deterministic noise in \cite{Nedic10}. Shi et al. \cite{Shi09} proposed a normalized incremental subgradient method, analyzed its convergence theory and demonstrated its application in wireless sensor networks.
Neto and Pierro \cite{NetoPierro09} explored the incremental subgradient method in a generic framework consisting of an optimality step and a feasibility step,
where both approximate subgradients and approximate projections are allowed, and illustrated its application in tomographic imaging.

Most papers in the literature of subgradient methods focus on convex optimization problems. Recently, much attention have been received beyond convex optimization. Quasi-convex optimization problems can be found in many important applications in various areas, such as economics, engineering, management science and various applied sciences; see  \cite{Avriel10,CrouzeixMM98,HaKS05} and references therein. The subgradient methods have been well extended to solve quasi-convex optimization problems, such as the standard subgradient method \cite{Kiwiel01}, inexact subgradient method \cite{HuEJOR15} and stochastic subgradient method \cite{HuJNCA16}. 
The convergence results of subgradient methods, in terms of objective values and iterates, for solving quasi-convex optimization problems have been well established under the H\"{o}lder condition of order $p$ and using the constant/diminishing/dynamic stepsize rules.
However, to the best of our knowledge, there is still no study devoted to investigating subgradient methods for solving the sum-minimization problem \eqref{eq-SP} for the case when each $f_i$ is quasi-convex.

The sum of ratios problem (in short, SOR) \cite{SchaibleSOR03} has a variety of important applications in economics and management science, such as multi-stage stochastic shipping \cite{AlmogySOR70}, government contracting \cite{ColantoniSOR69} and bond portfolio optimization \cite{KonnoSOR89}. Although some numerical methods, such as the branch and bound scheme \cite{BensonSOR10} and the interior point method \cite{FreundSOR2001}, have been studied to solve the SOR, most of them are computationally expensive to implement for large-scale problems or incapable to globally solve the SOR. The absence of effective numerical algorithms for large-scale problems hinders the research and applications of the SOR. Exploiting its structure, the SOR can be formulated as a sum-maximization problem of a number of quasi-concave component functions (see section 4 for the explanation), and so it is an important application of problem \eqref{eq-SP}. However, there is still no effective numerical algorithms for the large-scale sum-minimization problem \eqref{eq-SP} of quasi-convex functions, as well as the SOR.

To fill this gap, the aim of this paper is to develop the incremental subgradient methods for minimizing the sum of a number of quasi-convex component functions over a constraint set. In the remainder of this paper, we consider the sum-minimization problem \eqref{eq-SP} under the following hypothesis:
\begin{itemize}
  \item $f_i:\R^n\to \R$ is quasi-convex and continuous for each $i=1,\dots,m$, and $X\subseteq \R^n$ is nonempty, closed and convex.
\end{itemize}
Note that $f$ defined in \eqref{eq-SP}, the sum of quasi-convex component functions, is not necessarily to be quasi-convex.
The study of this paper is indeed beyond quasi-convex optimization, and so the direct application of standard subgradient method \cite{Kiwiel01} to solve problem \eqref{eq-SP} is not necessarily convergent.

Inspired by the idea of incremental approach, we propose a new incremental subgradient method to solve the sum-minimization problem \eqref{eq-SP}, which is different from the classical incremental subgradient method \eqref{eq-incre} in that the subgradient subiterations are only updated on the component functions whose minimal values are not achieved yet. Under a homogeneous assumption and the H\"{o}lder condition of order $p$ for the component functions, we provide a proper basic inequality and establish the convergence properties of the proposed incremental subgradient method when using the constant/diminishing/dynamic stepsize rules. The convergence properties are characterized in terms of function values and distances of iterates from the optimal solution set, and the finite convergence behavior to the optimality is further investigated when the solution set has a nonempty interior.

In the incremental subgradient method, the estimation of subgradients of all component functions at each iteration may be very expensive, especially when the number of component functions is large and no simple formulae for computing the subgradients exist. Note that the stochastic gradient descent algorithm is increasingly popular in large-scale machine learning problems; see \cite{Bertsekas2015,DuchiJMLR11,TanMaDQ2016,XiaoTZhang14} and references therein.
Employing the idea of stochastic gradient descent algorithm, we propose a randomized incremental subgradient method to save the computational cost of the incremental subgradient iteration, in which only one component function is randomly selected to construct the subgradient direction at each iteration.
The convergence results show that the randomized incremental subgradient method enjoys the convergence properties with probability 1 and achieves a better tolerance than that of the deterministic incremental subgradient method.

Furthermore, we formulate the SOR as a sum-minimization problem \eqref{eq-SP}, consider the multiple Cobb-Douglas productions efficiency problem (in short, MCDPE) \cite{BradleyFrey74} as an application of the SOR, and conduct numerical experiments on this problem via applying the proposed incremental subgradient methods. The numerical results show that the incremental subgradient methods are efficient for the MCDPE, especially for large-scale problems. This study may deliver a new approach for finding the global optimal solution of the large-scale SOR.

The paper is organized as follows. In section 2, we present the notations and preliminary results used in this paper. In section 3, we propose the deterministic and randomized incremental subgradient methods to solve the sum-minimization problem \eqref{eq-SP} and investigate their convergence properties when using the typical stepsize rules. The application to the SOR and the numerical study for the MCDPE are presented in section 4.

\section{Notations and preliminary results}
The notations used in this paper are standard; see, e.g., \cite{Bertsekas99,Bertsekas2015}.
We consider the $n$-dimensional Euclidean space $\R^n$ with inner product $\langle \cdot,\cdot\rangle$ and norm $\|\cdot\|$. For $x\in \R^n$ and $\delta\in\R_+$, we use $B(x,\delta)$ and $S(x,\delta)$ to denote the closed ball and the sphere of radius $\delta$ centered at $x$, respectively.
For $x\in \R^n$ and $Z\subseteq \R^n$, we write $\dist(x, Z)$ and $P_Z(x)$ to denote the Euclidean distance of $x$ from $Z$ and the Euclidean projection of $x$ onto $Z$, respectively, that is,
\begin{equation*}
\dist(x, Z):=\inf_{z\in Z}\|x-z\|\quad {\rm and} \quad P_Z(x):={\rm arg}\min_{z\in Z} \|x-z\|.
\end{equation*}
A function $h:\R^n\rightarrow \R$ is said to be quasi-convex if
\begin{equation*}
h((1-\alpha)x+\alpha y)\le \max\{h(x),h(y)\}\quad \mbox{for any $x,y\in \R^n$ and any $\alpha\in [0,1]$}.
\end{equation*}
For any $\alpha\in\R$, the level sets of $h$ are denoted by
\[
{\rm lev}_{<\alpha}h:=\{x\in \R^n :h(x)<\alpha\}\quad {\rm and} \quad {\rm lev}_{\le \alpha}h:=\{x\in \R^n :h(x)\le\alpha\}
\]
It is well-known that $h$ is quasi-convex if and only if ${\rm lev}_{<\alpha}h$ $\left({\rm and/or}~{\rm lev}_{\le \alpha}h\right)$ is convex for any $\alpha\in\R$. A function $h:\R^n\rightarrow \R$ is said to be coercive if $\lim_{\|x\|\to \infty} h(x)=\infty$, and so ${\rm lev}_{\le \alpha}h$ is bounded for any $\alpha\in\R$.

The subdifferential of quasi-convex function plays an important role in quasi-convex optimization. Several different types of subdifferentials of quasi-convex function have been introduced; see \cite{Aussel00,GrP73,HuEJOR15,Kiwiel01} and references therein. In particular, Kiwiel \cite{Kiwiel01} and Hu et al. \cite{HuEJOR15} introduced a notion of quasi-subdifferential defined a normal cone to ${\rm lev}_{<h(x)}h$, and used the related quasi-subgradient in their proposed subgradient methods. In the following definition, we recall the notion of quasi-subdifferential taken from \cite{HuEJOR15}.
\begin{definition}
Let $h:\R^n \to \R$ be a quasi-convex function, and let $x\in \R^n$. The quasi-subdifferential of $h$ at $x$ is defined by
\begin{equation}\label{eq-def-QS}
\partial h(x)=\left\{g:\langle g,y-x\rangle \le 0 \mbox{ for any } y\in {\rm lev}_{<h(x)}h\right\}.
\end{equation}
\end{definition}

The H\"{o}lder condition of order $p$ was used in \cite{Konnov94} to describe some properties of quasi-subgradient, and it plays an important role in the convergence study of subgradient-type methods for quasi-convex optimization problems \cite{HuEJOR15,HuJNCA16}.
\begin{definition}\label{def-HC}
Let $p>0$, $L>0$ and $x\in \R^n$. $h:\R^n\rightarrow \R$ is said to satisfy the H\"{o}lder condition of order $p$ with modulus $L$ at $x$ if
\begin{equation}\label{eq-HC}
\mbox{$|h(y)-h(x)|\le L\|y-x\|^{p}$\quad for any $y\in \R^n$}.
\end{equation}
$h$ is said to satisfy the H\"{o}lder condition of order $p$ with modulus $L$ on $X$ if \eqref{eq-HC} holds for any $x\in X$.
\end{definition}

The H\"{o}lder condition with order 1 is reduced to the Lipschitz condition, and this property holds for very broad classes of functions with various values of $p\in (0,+\infty)$. The following lemma recalls an important property of a quasi-convex function that satisfies the H\"{o}lder condition of order $p$. This property is a key to establish a basic inequality in convergence analysis of subgradient-type methods.
\begin{lemma}[{\cite[Proposition 2.1]{Konnov03}}]\label{lem-HC}
Let $h:\R^n\rightarrow \R$ be a quasi-convex and continuous function, $X$ be a closed and convex set, and let $X^*$ be the set of minima of $h$ over $X$.
Let $p>0$ and $L>0$, and suppose that $h$ satisfies the H\"{o}lder condition of order $p$ with modulus $L$ at some $x^*\in X^*$.
Then, for any $x\in X\setminus X^*$, it holds that
\[
\mbox{$h(x)-h(x^*)\le L \left\langle g,x-x^*\right\rangle^p$\quad for any $g\in \partial h(x) \cap S(0,1)$}.
\]
\end{lemma}

We end this section by recalling the following two lemmas, which are useful in the convergence analysis of incremental subgradient methods.
\begin{lemma}[{\cite[Lemma 4.1]{HuangYang03}}]\label{lem-IneQ-HYMOR}
Let $a_i\ge 0$ for $i=1,2,\dots,n$. The following relations hold.
\begin{enumerate}[{\rm (i)}]
  \item If $\gamma\in (0,1]$, then
  \[
  \Big(\sum_{i=1}^n a_i\Big)^\gamma\le \sum_{i=1}^n a_i^\gamma \le n\Big(\sum_{i=1}^n a_i\Big)^\gamma.
  \]
  \item If $\gamma\in [1,\infty)$, then
  \[
  \frac{1}{n^{\gamma-1}}\Big(\sum_{i=1}^n a_i\Big)^\gamma\le \sum_{i=1}^n a_i^\gamma \le \Big(\sum_{i=1}^n a_i\Big)^\gamma.
  \]
\end{enumerate}
\end{lemma}

\begin{lemma}[{\cite[Lemma 2.1]{Kiwiel04}}]\label{lem-averaging}
Let $\{a_k\}$ be a sequence of scalars, and let $\{v_k\}$ be a sequence of nonnegative scalars.
Suppose that $\lim_{k\to \infty} \sum_{i=1}^k v_i= \infty$. Then it holds that
\[
\liminf_{k\to \infty} a_k\le \liminf_{k\to \infty} \frac{\sum_{i=1}^k v_i a_i}{\sum_{i=1}^k v_i}
\le \limsup_{k\to \infty} \frac{\sum_{i=1}^k v_i a_i}{\sum_{i=1}^k v_i}\le \limsup_{k\to \infty}a_k.
 \]
In particular, if $\lim_{k\to \infty}a_k=a$, then $\lim_{k\to \infty} \frac{\sum_{i=1}^k v_i a_i}{\sum_{i=1}^k v_i}=a$.
\end{lemma}

\section{Incremental subgradient methods and convergence analysis}
In this section, we propose the incremental subgradient methods, including the deterministic and randomized styles, to solve problem \eqref{eq-SP}, and investigate their convergence properties when using typical stepsize rules.
We write $f^*$ and $X^*$ to denote the optimal value and the (global) optimal solution set of problem \eqref{eq-SP} respectively, that is,
\begin{equation*}\label{eq-min}
f^*:=\min_{x\in X}\sum_{i=1}^mf_i(x)\quad {\rm and}\quad X^*:={\rm arg}\min_{x\in X} \sum_{i=1}^mf_i(x),
\end{equation*}
and define
\[
f_i^*:=\min_{x\in X}f_i(x)\quad {\rm and}\quad  X_i^*:={\rm arg}\min_{x\in X} f_i(x)\quad \mbox{for } i=1,\dots,m.
\]
To accomplish the convergence analysis, the following two assumptions are made throughout this paper.

\begin{assumption}\label{asp-HF}
The component functions of problem \eqref{eq-SP} have a common optimal solution.
\end{assumption}

\begin{assumption}\label{asp-HC}
Let $p>0$ and $L_i>0$ for $i=1,\dots,m$. For each $i=1,\dots,m$, $f_i$ satisfies the H\"{o}lder condition of order $p$ with modulus $L_i$ on $X$.
\end{assumption}

The H\"{o}lder condition of order $p$ was assumed in \cite{HuEJOR15,HuJNCA16} to develop the convergence theory of several subgradient-type methods for quasi-convex optimization. Assumption \ref{asp-HC} consists of the H\"{o}lder condition of order $p$ for all component functions of problem \eqref{eq-SP}. Furthermore, we write
\begin{equation}\label{eq-lmax}
L_{\max}:=\max_{i=1,\dots,m} L_i.
\end{equation}
\begin{remark}
{\rm
It is easy to check that Assumption \ref{asp-HF} is equivalent to $X^*= \cap_{i=1}^m X_i^*\neq \emptyset$. Assumption \ref{asp-HF} is a homogeneous assumption for the component functions of problem \eqref{eq-SP}, and it also says that $f^*=\sum_{i=1}^m f_i^*$. Shi et al. \cite{Shi09} used Assumption \ref{asp-HF} to explore the convergence properties of a normalized incremental subgradient method for minimizing a sum of convex component functions. Although, under Assumption \ref{asp-HF}, we can approach the optimal value of problem \eqref{eq-SP} via minimizing component functions $f_i$ over $X$ separately, it is still difficult to find a common optimal solution, i.e., an optimal solution of problem \eqref{eq-SP}, which is an essential issue of decision-making problems. In this paper, we propose the incremental subgradient methods to resolve this issue.
}
\end{remark}

Below, we show a class of application problems, i.e., the quasi-convex feasibility problem, that can be formulated as the framework \eqref{eq-SP} and satisfy Assumptions \ref{asp-HF} and \ref{asp-HC}.
\begin{example}
The feasibility problem is at the core of the modeling of many problems in various areas of mathematics and physical sciences; see \cite{Bauschke96,HuSIOPT16} and references therein. In particular, the quasi-convex feasibility problem is an important class of feasibility problems (see \cite{CensorSegal06,GoffinLuoYe94,NimanaQuasiSFP2016}), which is to find a solution  of the following  system of inequalities:
\begin{equation}\label{eq-FP}
x\in X,\quad \mbox{and} \quad h_i(x)\le 0\quad  \mbox{for each }i=1,\dots,m,
\end{equation}
where $h_i:\R^n\to \R$ is quasi-convex and continuous for each $i=1,\dots,m$, and $X$ is nonempty, closed and convex. It is always assumed that the solution set of problem \eqref{eq-FP} is nonempty.
The feasibility problem \eqref{eq-FP} can be cast into the framework of sum-minimization problem \eqref{eq-SP} as the following model:
\begin{equation}\label{eq-FP-r}
\min_{x\in X} f(x):=\sum_{i=1}^m f_i(x),\quad \mbox{where $f_i:=\max\{h_i,0\}$ for each $i=1,\dots,m$}.
\end{equation}
It is clear that $f_i$ is quasi-convex if $h_i$ is quasi-convex for each $i=1,\dots,m$.
It is easy to see that Assumption \ref{asp-HF} is satisfied for problem \eqref{eq-FP-r} if the solution set of problem \eqref{eq-FP} is nonempty.
Assumption \ref{asp-HC} is satisfied for problem \eqref{eq-FP-r} if each function $h_i$ in \eqref{eq-FP} satisfies the H\"{o}lder condition of order $p$.
\end{example}

The stepsize rule has a critical effect on the convergence behavior and computational capacity of subgradient methods. In this paper, we investigate convergence properties of incremental subgradient methods by using the following typical stepsize rules.
\begin{enumerate}
\item[(S1)] \emph{Constant stepsize rule}:
\[v_k\equiv v\,(>0)\quad \mbox{for any } k\in \IN.\]

\item[(S2)] \emph{Diminishing stepsize rule}:
\begin{equation}\label{eq-dim+stepsize}
v_k>0,\quad \lim_{k\to \infty} v_k=0,\quad \sum_{k=0}^{\infty} v_k=\infty.
\end{equation}


\item[(S3)] \emph{Dynamic stepsize rule I}:
\begin{equation}\label{eq-dyn+stepsize}
v_k=\gamma_k \frac{C_{p,m}}{m^2}\left(f(x_k)-f^*\right)^{\frac1p}\quad \mbox{for any } k\in \IN,
\end{equation}
where $0<\underline{\gamma}\le \gamma_k \le \overline{\gamma}<2$ and
\begin{equation}\label{eq-para}
C_{p,m}:=L_{\max}^{-\frac{1}{p}}\min\left\{1,(2m)^{1-\frac1p}\right\}.
\end{equation}

\item[(S4)] \emph{Dynamic stepsize rule II}:
\begin{equation}\label{eq-dyn+stepsize-r}
v_k=\gamma_k \frac{R_{p,m}}{m}\left(f(x_k)-f^*\right)^{\frac1p}\quad \mbox{for any } k\in \IN,
\end{equation}
where $0<\underline{\gamma}\le \gamma_k \le \overline{\gamma}<2$ and
\begin{equation}\label{eq-para-r}
R_{p,m}:=L_{\max}^{-\frac1p}\min\left\{1,m^{1-\frac1p}\right\}.
\end{equation}
\end{enumerate}
Type (S3) is for the deterministic incremental subgradient method, while type (S4) is for the randomized incremental subgradient method.
Note that both types of dynamic stepsize rules are slightly difference from that of the classical incremental subgradient method (see \cite{Nedic01}).

\subsection{Incremental subgradient method}
The aim of this subsection is to propose an incremental subgradient method to solve problem \eqref{eq-SP}, and to study its convergence properties when using several different stepsize rules. The incremental subgradient method is formally described as follows.

\vspace{10pt}
{
\LinesNumbered
\IncMargin{1em}
\begin{algorithm}[H]
  \SetKwData{Left}{left}\SetKwData{This}{this}\SetKwData{Up}{up}
  \SetKwFunction{Union}{Union}\SetKwFunction{FindCompress}{FindCompress}
  \SetKwInOut{Input}{input}\SetKwInOut{Output}{output}
  \BlankLine
  Initialize an initial point $x_0\in \R^n$, a stepsize sequence $\{v_k\}$, and let $k:=0$\;
  \While{$f(x_k)>f^*$}
  {{Let $z_{k,0}:=x_k;$}\\
  \For{$i=1,\dots,m$}{
  {\eIf{$f_{i}(z_{k,i-1})=f_{i}^*$}
  {Let $z_{k,i}:=z_{k,i-1};$}
  {\emph{Calculate $g_{k,i}\in \partial f_i(z_{k,i-1}) \cap S(0,1)$, and let $z_{k,i}:=P_X \left(z_{k,i-1}-v_k g_{k,i}\right)$;}}
  }}
  Let $x_{k+1}:=z_{k,m}$ and $k:=k+1$.}
  \caption{Incremental subgradient method.}\label{alg-IncQS}
\end{algorithm}\
}

\begin{remark}
{\rm
Note that Algorithm \ref{alg-IncQS} is different from the classical incremental subgradient method for convex optimization in \cite{Nedic01} and the incremental gradient method for smooth optimization in \cite{Bertsekas99}. In particular, the classical incremental gradient/subgradient method \eqref{eq-incre} updates subgradient subiterations in a cyclic sequence on $\{1,\dots,m\}$; while Algorithm \ref{alg-IncQS} only updates subgradient subiterations in a cyclic sequence on the index set $\{i:f_{i}(z_{k,i-1})>f_{i}^*\}$, where the minimal value of $f_i$ is not achieved yet.
}
\end{remark}

The following example illustrates that Algorithm \ref{alg-IncQS} may not converge to the optimal value of problem \eqref{eq-SP} if the updated sequence on $\{i:f_{i}(z_{k,i-1})>f_{i}^*\}$ in Algorithm \ref{alg-IncQS} is replaced by a cyclic sequence on $\{1,\dots,m\}$ as in the classical incremental gradient/subgradient method \eqref{eq-incre}, even though Assumptions 1 and 2 are satisfied.

\begin{example}\label{exp-A1}
Consider problem \eqref{eq-SP}, where $X=\R$, $m=2$, and two component functions are
\[
f_1(x):=\max\{x,0\}\quad {\rm and}\quad f_2(x):=\max\{-x,0\}.
\]
Obviously, $f_1^*=f_2^*=0$, $X_1^*=\R_-$ and $X_2^*=\R_+$; $f(x):=f_1(x)+f_2(x)=|x|$, $f^*=0$ and $X^*=\{0\}$. It is easy to see that $X^*= X_1^* \cap X_2^*$, and that $f_1$ and $f_2$ are quasi-convex and satisfy the H\"{o}lder condition of order 1 on $\R$ with $L_{\max}=1$. Hence Assumptions \ref{asp-HF} and \ref{asp-HC} are satisfied.

In this setting, for any $x_0> 0$, one has that $\partial f_1(x_0)=\R_+$, $g_{1,1}=1$ and $z_{1,1}=x_0-v$. Note that
\[\partial f_2(x):=\left\{\begin{matrix}
   \R,&{x\ge 0,}\\ \R_-, &{x< 0.}
\end{matrix}\right.\]
Hence we can choose $g_{1,2}=-1$, and then $z_{1,2}=z_{1,1}+v=x_0$. That is, a fixed sequence is generated, and so $x_k\equiv x_0$ and $\lim_{k\to \infty} f(x_k)=x_0$. Therefore, the generated sequence does not converge to the optimal value/solution of problem \eqref{eq-SP}.
\end{example}

We now start the convergence analysis by providing a basic inequality of Algorithm \ref{alg-IncQS}, which shows a significant property of an incremental subgradient iteration. Recall that $C_{p,m}$ is defined in \eqref{eq-para}.
\begin{lemma}\label{lem-BI}
Suppose Assumptions \ref{asp-HF} and \ref{asp-HC} are satisfied. Let $\{x_k\}$ be a sequence generated by Algorithm \ref{alg-IncQS}.
Then, for any $x^*\in X^*$ and any $k\in \IN$, it holds that
\begin{equation}\label{eq-BI}
\|x_{k+1}-x^*\|^2\le \|x_k-x^*\|^2-2v_kC_{p,m} (f(x_k)-f^*)^{\frac{1}{p}}+m^2v_k^2.
\end{equation}
\end{lemma}
\begin{proof}
We first show that the following inequality holds for any $x^*\in X^*$, any $k\in \IN$ and $i=1,\dots,m$.
\begin{equation}\label{eq-lem-BI-1}
\|z_{k,i}-x^*\|^2 \le \|z_{k,i-1}-x^*\|^2-2v_k L_{\max}^{-\frac{1}{p}}\left(f_i(z_{k,i-1})-f_i^*\right)^{\frac{1}{p}}+ v_k^2.
\end{equation}
In view of Algorithm \ref{alg-IncQS}, if $f_{i}(z_{k,i-1})=f_{i}^*$, then it is updated that $z_{k,i}=z_{k,i-1}$, and so \eqref{eq-lem-BI-1} automatically holds; otherwise, $z_{k,i-1}\notin X_i^*$, and then one sees from Algorithm \ref{alg-IncQS} that
\begin{equation}\label{eq-lem-BI-1b}
z_{k,i}=P_X \left(z_{k,i-1}-v_k g_{k,i}\right).
\end{equation}
Note by Assumption \ref{asp-HF} that $x^*\in X^*= \cap_{i=1}^m X_i^*$, and so $x^*\in X_i^*$ for $i=1,\dots,m$. By the assumption that $z_{k,i-1}\notin X_i^*$, one has that $f(z_{k,i-1})> f_i^*$.
Then Lemma \ref{lem-HC} is applicable (with $f_i$, $z_{k,i-1}$, $X_i^*$ in place of $h$, $x$, $X^*$) to concluding that
\begin{equation}\label{eq-lem-BI-0a}
f_i(z_{k,i-1})-f_i^*=f_i(z_{k,i-1})-f_i(x^*)\le L_i\langle g_{k,i}, z_{k,i-1}-x^*\rangle^p\le L_{\max}\langle g_{k,i}, z_{k,i-1}-x^*\rangle^p
\end{equation}
(due to \eqref{eq-lmax}). By the nonexpansive property of the projection operator, it follows from \eqref{eq-lem-BI-1b} that
\begin{equation*}
\begin{array}{lll}
\|z_{k,i}-x^*\|^2&\le \left\|z_{k,i-1}-v_k g_{k,i}-x^*\right\|^2 \\
&= \|z_{k,i-1}-x^*\|^2-2v_k \left\langle g_{k,i}, z_{k,i-1}-x^*\right\rangle+ v_k^2\\
&\le \|z_{k,i-1}-x^*\|^2-2v_k L_{\max}^{-\frac{1}{p}}\left(f_i(z_{k,i-1})-f_i^*\right)^{\frac{1}{p}}+ v_k^2,
\end{array}
\end{equation*}
where the last inequality follows from \eqref{eq-lem-BI-0a}. Hence \eqref{eq-lem-BI-1} is proved.

Next, we estimate the second term in the right hand side of \eqref{eq-lem-BI-1} in terms of $f_i(x_k)-f_i^*$.
By Lemma \ref{lem-IneQ-HYMOR} (with $f_i(z_{k,i-1})-f_i^*$, $|f_i(z_{k,i-1})-f_i(x_k)|$, $\frac{1}{p}$ in place of $a_1$, $a_2$, $\gamma$), one has that
\begin{equation*}\label{eq-lem-BI-1a}
\begin{array}{llll}
(f_i(x_k)-f_i^*)^{\frac{1}{p}}&\le \left((f_i(z_{k,i-1})-f_i^*)+|f_i(z_{k,i-1})-f_i(x_k)|\right)^{\frac{1}{p}}\\
&\le \max\left\{1,2^{\frac{1}{p}-1}\right\}\left((f_i(z_{k,i-1})-f_i^*)^{\frac{1}{p}}+|f_i(z_{k,i-1})-f_i(x_k)|^{\frac{1}{p}}\right).
\end{array}
\end{equation*}
Denoting
\begin{equation}\label{eq-lem-BI-1c}
C_p:=\left(\max\left\{1,2^{\frac{1}{p}-1}\right\}\right)^{-1}=\min\left\{1,2^{1-\frac{1}{p}}\right\},
\end{equation}
the above inequality is reduced to
\begin{equation}\label{eq-lem-BI-2}
(f_i(z_{k,i-1})-f_i^*)^{\frac{1}{p}}\ge C_p(f_i(x_k)-f_i^*)^{\frac{1}{p}}-|f_i(z_{k,i-1})-f_i(x_k)|^{\frac{1}{p}}.
\end{equation}
By Assumption \ref{asp-HC} (cf. \eqref{eq-HC}) and in view of Algorithm \ref{alg-IncQS},
it follows that
\begin{equation*}
|f_i(z_{k,i-1})-f_i(x_k)|\le L_i\|z_{k,i-1}-x_k\|^p\le L_{\max}\left(\sum_{j=1}^{i-1}\|z_{k,j}-z_{k,j-1}\|\right)^p\le L_{\max}\left(v_k(i-1)\right)^p.
\end{equation*}
Hence \eqref{eq-lem-BI-2} is reduced to
\[
(f_i(z_{k,i-1})-f_i^*)^{\frac{1}{p}}\ge C_p(f_i(x_k)-f_i^*)^{\frac{1}{p}}-L_{\max}^\frac1pv_k(i-1),
\]
and so \eqref{eq-lem-BI-1} yields that
\begin{equation}\label{eq-lem-BI-3}
\|z_{k,i}-x^*\|^2\le \|z_{k,i-1}-x^*\|^2-2v_k L_{\max}^{-\frac{1}{p}}C_p\left(f_i(x_k)-f_i^*\right)^{\frac{1}{p}}+(2i-1)v_k^2.
\end{equation}

Finally, summing \eqref{eq-lem-BI-3} over $i=1,\dots,m$, we obtain that
\begin{equation}\label{eq-lem-BI-4}
\|x_{k+1}-x^*\|^2\le \|x_k-x^*\|^2-2v_k L_{\max}^{-\frac{1}{p}}C_p \sum_{i=1}^m\left(f_i(x_k)-f_i^*\right)^{\frac{1}{p}}+m^2v_k^2.
\end{equation}
Note by Lemma \ref{lem-IneQ-HYMOR} (with $f_i(x_k)-f_i^*$ and $\frac{1}{p}$ in place of $a_i$ and $\gamma$) that
\[
\sum_{i=1}^m(f_i(x_k)-f_i^*)^{\frac{1}{p}}
\ge \min\left\{1,m^{1-\frac{1}{p}}\right\}\left(\sum_{i=1}^m (f_i(x_k)-f_i^*)\right)^{\frac{1}{p}}
=\min\left\{1,m^{1-\frac{1}{p}}\right\}(f(x_k)-f^*)^{\frac{1}{p}}
\]
(thanks to Assumption 1), and by \eqref{eq-lem-BI-1c} that
\[
L_{\max}^{-\frac{1}{p}}C_p\min\left\{1,m^{1-\frac{1}{p}}\right\}=
L_{\max}^{-\frac{1}{p}}\min\left\{1,2^{1-\frac{1}{p}}\right\}\,\min\left\{1,m^{1-\frac{1}{p}}\right\}
=C_{p,m}
\]
(cf. \eqref{eq-para}). Therefore, \eqref{eq-BI} is seen to hold by \eqref{eq-lem-BI-4}. The proof is complete.
\end{proof}

\begin{remark}
In Algorithm \ref{alg-IncQS}, the subgradient subiteration is processed in an ordered cyclic sequence on the index set $\{i:f_{i}(z_{k,i-1})>f_{i}^*\}$.
It is worthy mentioning that the proof of Lemma \ref{lem-BI}, as well as the convergence analysis of Algorithm \ref{alg-IncQS}, still work if any order of  component functions is assumed, as long as each component on $\{i:f_{i}(z_{k,i-1})>f_{i}^*\}$ is taken into account exactly once within a cycle.
Hence, in applications, we could reorder the components $f_i$ by either shifting or reshuffling at the beginning of each cycle, and then proceed with the calculations until the end of this cycle.
\end{remark}

By virtue of Lemma \ref{lem-BI}, we establish the convergence results of the incremental subgradient method when using different stepsize rules in Theorems \ref{thm-const}-\ref{thm-dyn}, respectively.

\begin{theorem}\label{thm-const}
Suppose Assumptions \ref{asp-HF} and \ref{asp-HC} are satisfied. Let $\{x_k\}$ be a sequence generated by Algorithm \ref{alg-IncQS} with the constant stepsize rule {\rm (S1)}. Then we have
\begin{equation}\label{eq-thm-cons}
\liminf_{k\to \infty} f(x_k)\le f^*+\left(\frac{m^2v}{2C_{p,m}}\right)^p.
\end{equation}
\end{theorem}
\begin{proof}
We prove by contradiction, assuming that
\begin{equation*}
\liminf_{k\to \infty} f(x_k)> f^*+\left(\frac{m^2v}{2C_{p,m}}\right)^p.
\end{equation*}
Consequently, there exist $\delta>0$ and $N\in \IN$ such that
\begin{equation*}
f(x_k)> f^*+\left(\frac{m^2v}{2C_{p,m}}+\delta\right)^p\quad \mbox{for any } k\ge N.
\end{equation*}
Therefore, it follows from Lemma \ref{lem-BI} that for any $k\ge N$
\begin{equation*}\label{eq-thm-cons-1}
\|x_{k+1}-x^*\|^2\le \|x_k-x^*\|^2-2vC_{p,m}(f(x_k)-f^*)^{\frac{1}{p}}+m^2v^2
 < \|x_k-x^*\|^2-2v\delta C_{p,m}.
\end{equation*}
Summing the above inequality over $k=N,\dots,t-1$, we have
\[
\|x_{t}-x^*\|^2<\|x_{N}-x^*\|^2-2(t-N)v\delta C_{p,m},
\]
which yields a contradiction for a sufficiently large $t$. The proof is complete.
\end{proof}

\begin{remark}
{\rm
Theorem \ref{thm-const} shows the convergence of Algorithm \ref{alg-IncQS} to the optimal value of problem \eqref{eq-SP} within a tolerance when the constant stepsize rule is adopted.
The tolerance in \eqref{eq-thm-cons} is given in terms of the stepsize and circumstances of problem \eqref{eq-SP}, including the number of component functions and parameters of H${\rm \ddot{o}}$lder conditions.
In particular, when $m=1$, problem \eqref{eq-SP} is reduced to a constrained quasi-convex optimization problem, and then the convergence result described in Theorem \ref{thm-const} is reduced to that of \cite[Theorem 3.1]{HuEJOR15} (when noise and error are vanished).
When each component function in problem \eqref{eq-SP} is convex, the H${\rm \ddot{o}}$lder condition ($p=1$) is equivalent to the bounded subgradient assumption, and then the convergence result described in Theorem \ref{thm-const} is reduced to that of \cite[Proposition 2.1]{Nedic01}.
}
\end{remark}

\begin{theorem}\label{thm-dimi}
Suppose Assumptions \ref{asp-HF} and \ref{asp-HC} are satisfied. Let $\{x_k\}$ be a sequence generated by Algorithm \ref{alg-IncQS} with the diminishing stepsize rule {\rm (S2)}. Then the following statements hold.
\begin{enumerate}[{\rm (i)}]
  \item $\liminf_{k\to \infty} f(x_k)= f^*$.
  \item If $f$ is coercive, then $\lim_{k\to \infty} f(x_k)= f^*$ and $\lim_{k\to \infty} \dist(x_k,X^*)=0$.
  \item If $\sum_{k=0}^\infty v_k^2<\infty$, then $\{x_k\}$ converges to an optimal solution of problem \eqref{eq-SP}.
\end{enumerate}
\end{theorem}
\begin{proof} 
(i) Fix $x^*\in X^*$. Summing \eqref{eq-BI} over $k=0,1,\dots,n-1$, we have
\begin{equation}\label{eq-thm-dimi-0}
\|x_n-x^*\|^2\le \|x_0-x^*\|^2-2C_{p,m} \sum_{k=0}^{n-1} v_k(f(x_k)-f^*)^{\frac{1}{p}}+m^2\sum_{k=0}^{n-1}v_k^2,
\end{equation}
and thus
\begin{equation}\label{eq-thm-dimi-0a}
\frac{\sum_{k=0}^{n-1} v_k(f(x_k)-f^*)^{\frac{1}{p}}}{\sum_{k=0}^{n-1}v_k}\le \frac{\|x_0-x^*\|^2}{2C_{p,m}\sum_{k=0}^{n-1}v_k}+\frac{m^2\sum_{k=0}^{n-1}v_k^2}{2C_{p,m}\sum_{k=0}^{n-1}v_k}.
\end{equation}
Note by \eqref{eq-dim+stepsize} that
\begin{equation}\label{eq-thm-dimi-0b}
\lim_{n\to \infty} \frac{\|x_0-x^*\|^2}{\sum_{k=0}^{n-1}v_k}=0,
\end{equation}
and by Lemma \ref{lem-averaging} (with $v_k$ in place of $a_k$) that
\begin{equation}\label{eq-thm-dimi-0c}
\lim_{n\to \infty} \frac{\sum_{k=0}^{n-1}v_k^2}{\sum_{k=0}^{n-1}v_k}= \lim_{n\to \infty} v_n =0.
\end{equation}
Then, by Lemma \ref{lem-averaging} (with $(f(x_k)-f^*)^{\frac{1}{p}}$ in place of $a_k$), \eqref{eq-thm-dimi-0a} implies that
\[
\liminf_{n\to \infty}\, (f(x_k)-f^*)^{\frac{1}{p}}\le \liminf_{n\to \infty} \frac{\sum_{k=0}^{n-1}v_k(f(x_k)-f^*)^{\frac{1}{p}}}{\sum_{k=0}^{n-1}v_k}
\le 0.
\]
This shows the desired assertion.

(ii) Fix $\sigma>0$. Since $\{v_k\}$ is diminishing, there exists $N\in \IN$ be such that
\begin{equation}\label{eq-dimi-1}
v_k\le \frac1{m^2}C_{p,m}\sigma^{\frac1p}\quad \mbox{for any } k\ge N.
\end{equation}
Define
\begin{equation}\label{eq-abs-dimi-b}
X_\sigma:=X\cap {\rm lev}_{\le f^*+\sigma}f\quad {\rm and}\quad \rho(\sigma):=\max_{x\in X_\sigma}\dist(x,X^*).
\end{equation}
By the assumption that $f$ is coercive, it follows that its sublevel set ${\rm lev}_{\le f_*+\sigma}f$ is bounded, and so is $X_\sigma$. Thus, by \eqref{eq-abs-dimi-b}, one has $ \rho(\sigma)<\infty$.
Fix $k\ge N$. We show
\begin{equation}\label{eq-imply}
\dist(x_{k+1},X^*)\le \max\{\dist(x_{k},X^*), \rho(\sigma)+\frac1{m}C_{p,m}\sigma^{\frac1p}\}
\end{equation}
by claiming the following two implications:
\begin{equation}\label{eq-imply-1}
[f(x_k)>f_*+\sigma]\quad \Rightarrow \quad [\dist(x_{k+1},X^*)\le \dist(x_{k},X^*)];
\end{equation}
\begin{equation}\label{eq-imply-2}
[f(x_k)\le f_*+\sigma]\quad \Rightarrow \quad [\dist(x_{k+1},X^*)\le \rho(\sigma)+\frac1{m}C_{p,m}\sigma^{\frac1p}].
\end{equation}
To prove \eqref{eq-imply-1}, we suppose that $f(x_k)>f_*+\sigma$. Then $x_k\notin X^*$, and so Lemma \ref{lem-BI} is applicable to concluding that, for any $x^*\in X^*$,
\[
\|x_{k+1}-x^*\|^2\le \|x_k-x^*\|^2-2C_{p,m}v_k \sigma^{\frac{1}{p}}+m^2v_k^2\le \|x_k-x^*\|^2-m^2v_k^2
\]
(due to \eqref{eq-dimi-1}). Consequently, one can prove \eqref{eq-imply-1} by selecting $x^*:=P_{X^*}(x_{k+1})$. To show \eqref{eq-imply-2}, we suppose that $f(x_k)\le f_*+\sigma$. Then we conclude that $x_k\in X_\sigma$, and so, \eqref{eq-abs-dimi-b} says that $\dist(x_k,X^*)\le \rho(\sigma)$. In view of Algorithm \ref{alg-IncQS}, for any $x^*\in X^*$, we obtain
\[
\|x_{k+1}-x^*\|\le \|x_{k}-x^*\|+\sum_{i=1}^m\|z_{k,i}-z_{k,i-1}\|\le \|x_{k}-x^*\|+v_km,
\]
and thus
\begin{equation*}\label{eq-thm-dimi-2}
\dist(x_{k+1},X^*)\le \dist(x_k,X^*) +v_km\le \rho(\sigma)+v_km.
\end{equation*}
This, together with \eqref{eq-dimi-1}, shows \eqref{eq-imply-2}. Therefore, \eqref{eq-imply} is proved as desired.

By assertion (i), we can assume, without loss of generality, that $f(x_N)\le f_*+\sigma$ (otherwise, we can choose a larger $N$); consequently, one has by \eqref{eq-imply-2} that $\dist(x_{N+1},X^*)\le \rho(\sigma)+\frac1{m}C_{p,m}\sigma^{\frac1p}$. Then, we inductively obtain by \eqref{eq-imply} that
\begin{equation}\label{eq-thm-dimi-2a}
\dist(x_{k},X^*)\le \rho(\sigma)+\frac1{m}C_{p,m}\sigma^{\frac1p}\quad \mbox{for any } k> N.
\end{equation}
Since $f$ is continuous and coercive, its sublevel sets are compact, and so, it is trivial to see that $\lim_{\sigma\to 0}\rho(\sigma)=0$.
Hence we obtain by \eqref{eq-thm-dimi-2a} that $\lim_{k\to \infty}\dist(x_{k},X^*)=0$, and thus $\lim_{k\to \infty} f(x_k)= f_*$ (by the continuity of $f$).

(iii) By the assumption that $\sum_{k=0}^\infty v_k^2<\infty$, one sees from \eqref{eq-thm-dimi-0} that $\{\|x_k-x^*\|\}$ is bounded, and so is $\{x_k\}$. Since further it was proved in assertion (i) of this theorem that $\liminf_{k\to \infty} f(x_k)= f^*$, it follows that $\{x_k\}$ has at least a cluster point falling in $X^*$, assumed as $\bar{x}\in X^*$. Noting that $\lim_{n\to \infty}\sum_{k=n}^\infty v_k^2=0$, we obtain by \eqref{eq-BI} (with $\bar{x}$ in place of $x^*$) that $\{\|x_k-\bar{x}\|^2\}$ is a Cauchy sequence, and thus it converges to 0. Hence $\{x_k\}$ converges to $\bar{x}\,(\in X^*)$. The proof is complete.
\end{proof}

It was reported in \cite{HuEJOR15,HuJNCA16} that the H${\rm \ddot{o}}$lder condition of order $p$ (i.e., Assumption \ref{asp-HC}) is essential for the convergence behavior of subgradient-type methods for quasi-convex optimization. The following example illustrates that Assumption \ref{asp-HF} is also essential for the validity of the established convergence theorems.
\begin{example}\label{exp-HC}
Consider problem \eqref{eq-SP}, where $X=\R$, $m=2$, and two component functions are
\[
f_1(x):=\max\{x+2,0\}\quad {\rm and}\quad f_2(x):=\max\{-2x+2,0\}.
\]
Obviously, $f_1^*=f_2^*=0$, $X_1^*=(-\infty,-2]$ and $X_2^*=[1,+\infty)$; $f^*=3$ and $X^*=\{1\}$. Clearly, one sees that $X^*\neq X_1^*\cap X_2^*$, and so Assumption \ref{asp-HF} is not satisfied. It is easy to verify that $f_1$ and $f_2$ are quasi-convex and satisfy the H\"{o}lder condition of order 1 on $\R$ with $L_{\max}=2$, and so Assumption \ref{asp-HC} is satisfied.

Starting from $x_0=0$, we apply Algorithm \ref{alg-IncQS} to solve problem \eqref{eq-SP}. We claim that the generated sequence may not converge to the optimal value/solution of problem \eqref{eq-SP} for any stepsize. Indeed, in this setting, one has that $\partial f_1(0)=\R_+$, $g_{1,1}=1$ and $z_{1,1}=-v<0$, and then $\partial f_2(z_{k,1})=\R_-$, $g_{1,2}=-1$ and $z_{1,2}=z_{1,1}+v=0$.
Consequently, a fixed sequence is generated, and so $x_k\equiv 0$ and $\lim_{k\to \infty} f(x_k)=4$. Hence, Theorem \ref{thm-const} fails whenever $v<\frac14$, and Theorem \ref{thm-dimi} fails for any diminishing stepsize.
\end{example}

When the prior information on $f^*$ is available, a dynamic stepsize rule is usually considered to achieve an optimal convergence property in the literature of subgradient methods; see, e.g., \cite{Bertsekas99,CensorSegal06,HuJNCA16,Nedic01,Nedic10}. Below, we show the convergence of the incremental subgradient method
to an optimal solution of problem \eqref{eq-SP} when the dynamic stepsize rule (S3) is adopted.

\begin{theorem}\label{thm-dyn}
Suppose Assumptions \ref{asp-HF} and \ref{asp-HC} are satisfied. Let $\{x_k\}$ be a sequence generated by Algorithm \ref{alg-IncQS} with the dynamic stepsize rule {\rm (S3)}. Then $\{x_k\}$ converges to an optimal solution of problem \eqref{eq-SP}.
\end{theorem}
\begin{proof}
By Lemma \ref{lem-BI} and \eqref{eq-dyn+stepsize}, we obtain that, for any $x^*\in X^*$ and any $k\in \IN$,
\begin{equation*}\label{eq-thm-dyn}
\begin{array}{lll}
\|x_{k+1}-x^*\|^2&\le \|x_k-x^*\|^2-2v_kC_{p,m} (f(x_k)-f^*)^{\frac{1}{p}}+m^2v_k^2\\
&= \|x_k-x^*\|^2-\gamma_k(2-\gamma_k) \frac{C_{p,m}^2}{m^2}(f(x_k)-f^*)^{\frac{2}{p}}\\
&\le \|x_k-x^*\|^2-\underline{\gamma}(2-\overline{\gamma}) \frac{C_{p,m}^2}{m^2}(f(x_k)-f^*)^{\frac{2}{p}}.
\end{array}
\end{equation*}
This shows that the sequence $\{\|x_k-x^*\|\}$ is decreasing, and hence $\{x_k\}$ is bounded. It also follows from the above inequality that
\[
\sum_{k=1}^\infty (f(x_k)-f^*)^{\frac2p} \le \frac{m^2}{\underline{\gamma}(2-\overline{\gamma})C_{p,m}^2}\|x_0-x^*\|^2,
\]
which is finite. Consequently, noting that $f(x_k)-f^*\ge 0$ for any $k\in \IN$, one has that $\lim_{k\to \infty} f(x_k)= f^*$.
Hence, any cluster point of $\{x_k\}$ is an optimal solution of problem \eqref{eq-SP}, denoted by $\bar{x}\in X^*$.
Since further the sequence $\{\|x_k-x^*\|\}$ is decreasing, it converges to $\|\bar{x}-x^*\|$ for any $x^*\in X^*$. Hence $\{x_k\}$ converges to an optimal solution of problem \eqref{eq-SP}. The proof is complete.
\end{proof}

At the end of this subsection, we present a finite convergence property of the incremental subgradient method to the solution set $X^*$ of problem \eqref{eq-SP} under the assumption that $X^*$ has a nonempty interior.

\begin{theorem}\label{thm-FC}
Suppose Assumptions \ref{asp-HF} and \ref{asp-HC} are satisfied. Let $\{x_k\}$ be a sequence generated by Algorithm \ref{alg-IncQS}. Let $x^*\in X^*$ and $\sigma>0$, and suppose that $\mathbf{B}(x^*,\sigma)\subseteq X^*$.
Then $x_k\in X^*$ for some $k\in \IN$, provided that one of the following conditions:
\begin{enumerate}[{\rm (i)}]
  \item $v_k = v\in (0, \frac{2\sigma}{m})$ for any $k\in \IN$, and
  \item the sequence $\{v_k\}$ satisfies the diminishing stepsize rule {\rm (S2)}.
\end{enumerate}
\end{theorem}
\begin{proof}
To proceed, we define a new process $\{\hat{x}_k\}$ via the classical incremental subgradient method starting with $\hat{x}_0:=x_0$. That is, for each iteration, we start with $\hat{z}_{k,0}:=\hat{x}_k$, through $i=1,\dots,m$,
\begin{equation}\label{eq-newprog}
\hat{z}_{k,i}:=P_X \big(\hat{z}_{k,i-1}-v_k \hat{g}_{k,i}\big), \mbox{ where }
\hat{g}_{k,i}\in\left\{\begin{matrix}
   \partial f_i(\hat{z}_{k,i-1}) \cap S(0,1),&{\mbox{if }f_{i}(\hat{z}_{k,i-1})>f_{i}^*,}\\ \{0\}, &{\rm otherwise},
\end{matrix}\right.
\end{equation}
and finally arrive at $\hat{x}_{k+1}:=\hat{z}_{k,m}$.
Comparing with Algorithm \ref{alg-IncQS}, we observe that the process $\{\hat{x}_k\}$ is identical to $\{x_k\}$.

We prove by contradiction, assuming that $f(\hat{x}_k)>f^*$ for any $k\in \mathbb{N}$. Fixing $k\in \mathbb{N}$, we define
\begin{equation}\label{eq-thm-FC-Ik}
I_k:=\{i\in \{1,\dots,m\}:f_{i}(\hat{z}_{k,i-1})>f_{i}^*\}.
\end{equation}
Clearly, $I_k\neq \emptyset$; otherwise, $f(\hat{x}_k)=f^*$ and a contradiction is achieved. Fix $i\in I_k$. By the assumption that $\mathbf{B}(x^*,\sigma)\subseteq X^*$ and $\|\hat{g}_{k,i}\|= 1$, one has that $x^*+\sigma\hat{g}_{k,i}\in X^*$, and hence $f_i(x^*+\sigma\hat{g}_{k,i})=f_i^*<f_{i}(\hat{z}_{k,i-1})$ (cf. \eqref{eq-thm-FC-Ik}). Then it follows from Definition \ref{eq-def-QS} that
$\langle \hat{g}_{k,i}, x^*+\sigma \hat{g}_{k,i}-\hat{z}_{k,i-1}\rangle \le 0$. Consequently,
\begin{equation*}\label{eq-thm-FC-1}
\langle \hat{g}_{k,i}, \hat{z}_{k,i-1}- x^*\rangle \ge \sigma\mbox{ when }i\in I_k, \mbox{ and }
\langle \hat{g}_{k,i}, \hat{z}_{k,i-1}- x^*\rangle =0 \mbox{ otherwise}
\end{equation*}
(by \eqref{eq-newprog} that $\hat{g}_{k,i}=0$ when $i\notin I_k$).
Therefore, we obtain that
\begin{equation}\label{eq-thm-FC-1b}
\sum_{i=1}^m \langle \hat{g}_{k,i}, \hat{z}_{k,i-1}-x^*\rangle\ge |I_k|\sigma\ge \sigma\quad \mbox{for any }k \in \IN,
\end{equation}
where $|I_k|\ge 1$ since $I_k\neq \emptyset$.

On the other hand, by \eqref{eq-newprog}, it follows that
\[
\begin{array}{lll}
\|\hat{z}_{k,i}-x^*\|^2&\le \left\|\hat{z}_{k,i-1}-v_k \hat{g}_{k,i}-x^*\right\|^2 \\
&\le \|\hat{z}_{k,i-1}-x^*\|^2-2v_k \left\langle \hat{g}_{k,i}, \hat{z}_{k,i-1}-x^*\right\rangle+ v_k^2.
\end{array}
\]
Summing the above inequality over $i=1,\dots,m$, one has
\begin{equation*}\label{eq-FC-2a}
v_k\sum_{i=1}^m \langle \hat{g}_{k,i}, \hat{z}_{k,i-1}-x^*\rangle
\le \frac{\|\hat{x}_{k+1}-x^*\|^2-\|\hat{x}_{k}-x^*\|^2}{2}+\frac{mv_k^2}{2};
\end{equation*}
consequently,
\begin{equation}\label{eq-FC-3}
\frac{\sum_{k=0}^{n-1} \left(v_k\sum_{i=1}^m \langle \hat{g}_{k,i}, \hat{z}_{k,i-1}-x^*\rangle\right)}{\sum_{k=0}^{n-1} v_k}
\le \frac{\|x_0-x^*\|^2}{2\sum_{k=0}^{n-1} v_k}+\frac{m\sum_{k=0}^{n-1} v_k^2}{2\sum_{k=0}^{n-1} v_k}.
\end{equation}
We now claim, under the assumption of (i) or (ii), that
\begin{equation}\label{eq-FC-4}
\liminf_{n\to \infty}\, \sum_{i=1}^m \langle \hat{g}_{n,i}, \hat{z}_{n,i-1}-x^*\rangle <\sigma.
\end{equation}
\begin{enumerate}[{\rm (i)}]
\item When a constant stepsize $v\in (0, \frac{2\sigma}{m})$ is used, \eqref{eq-FC-3} is reduced to
\[
\frac{\sum_{k=0}^{n-1} \sum_{i=1}^m \langle \hat{g}_{k,i}, \hat{z}_{k,i-1}-x^*\rangle}{n} \le \frac{\|x_0-x^*\|^2}{2nv}+\frac{mv}2,
\]
and thus, by Lemma \ref{lem-averaging}, we obtain that
\[\liminf_{n\to \infty} \sum_{i=1}^m \langle \hat{g}_{n,i}, \hat{z}_{n,i-1}-x^*\rangle\le \liminf_{n\to \infty} \frac{\sum_{k=0}^{n-1} \sum_{i=1}^m \langle \hat{g}_{k,i}, \hat{z}_{k,i-1}-x^*\rangle}{n}\le \frac{mv}2<\sigma.\]
This shows \eqref{eq-FC-4}, as desired.
\item When a diminishing stepsize is used, by \eqref{eq-thm-dimi-0b} and \eqref{eq-thm-dimi-0c}, it also follows from Lemma \ref{lem-averaging} and \eqref{eq-FC-3} that
\[\liminf_{n\to \infty} \sum_{i=1}^m \langle \hat{g}_{n,i}, \hat{z}_{n,i-1}-x^*\rangle\le
\liminf_{n\to \infty} \frac{\sum_{k=0}^{n-1} \left(v_k\sum_{i=1}^m \langle \hat{g}_{k,i}, \hat{z}_{k,i-1}-x^*\rangle\right)}{\sum_{k=0}^{n-1} v_k}
\le 0<\sigma.\]
\end{enumerate}
Hence we proved \eqref{eq-FC-4} under the assumption of (i) or (ii), which arrives at a contradiction with \eqref{eq-thm-FC-1b}.
The proof is complete.
\end{proof}

\subsection{Randomized incremental subgradient method}
It could be very computationally expensive to calculate the subgradients of all component functions at each iteration of the incremental subgradient method, especially when the number of component functions is large and the calculation of subgradients is not simple.
To economize on the computational cost of each iteration, we propose a randomized incremental subgradient method, in which only one component function $f_{\omega_i}$ is randomly selected to construct the subgradient direction at each iteration, rather than to take each $f_i$ into account exactly once within an ordered cycle.

This subsection aims to explore the convergence properties of the randomized incremental subgradient method for solving problem \eqref{eq-SP} when using typical stepsize rules. The randomized incremental subgradient method is formally presented as follows.

\vspace{10pt}
{
\LinesNumbered
\IncMargin{1em}
\begin{algorithm}[H]
  \SetKwData{Left}{left}\SetKwData{This}{this}\SetKwData{Up}{up}
  \SetKwFunction{Union}{Union}\SetKwFunction{FindCompress}{FindCompress}
  \SetKwInOut{Input}{input}\SetKwInOut{Output}{output}
  \BlankLine
  Initialize an initial point $x_0\in \R^n$, a stepsize sequence $\{v_k\}$, and let $k:=0$\;
  \While{$f(x_k)>f^*$}
  {{Let $I_k:=\left\{i\in \{1,\dots,m\}:f_i(x_k)>f_i^*\right\};$}\\
  {Pick up equiprobably a random variable $\omega_k$ from the set $I_k$, calculate $g_{k,\omega_k}\in \partial f_{\omega_k}(x_k) \cap S(0,1)$, and let $x_{k+1}:=P_X \left(x_k-v_k g_{k,\omega_k}\right)$;}\\
  Let $k:=k+1$.}
  \caption{Randomized incremental subgradient method.}\label{alg-Rand-QS}
\end{algorithm}\
}

We recall the supermartingale convergence theorem, which is useful in the convergence analysis of the randomized incremental subgradient method.

\begin{theorem}[{\cite[pp. 148]{BertsekasTsitsiklis96}}]\label{thm-super}
Let $\{Y_k\}$, $\{Z_k\}$ and $\{W_k\}$ be three sequences of random variables, and let $\{\mathcal{F}_k\}$ be a sequence of sets of random variables such that $\mathcal{F}_k\subseteq \mathcal{F}_{k+1}$ for any $k\in \IN$. Suppose for any $k\in \IN$ that
\begin{enumerate}[{\rm (a)}]
  \item $Y_k$, $Z_k$ and $W_k$ are functions of nonnegative random variables in $\mathcal{F}_k$;
  \item $\mathbf{E}\left\{Y_{k+1}\mid \mathcal{F}_k\right\}\le Y_k-Z_k+W_k$;
  \item $\sum_{k=0}^\infty W_k<\infty$.
\end{enumerate}
Then $\sum_{k=0}^\infty Z_k<\infty$ and $\{Y_k\}$ converges to a nonnegative random variable $Y$ with probability 1.
\end{theorem}

To begin the convergence analysis of Algorithm \ref{alg-Rand-QS}, we provide below a basic inequality of a randomized incremental subgradient iteration in terms of expectation. Recall that $R_{p,m}$ is defined in \eqref{eq-para-r}.
\begin{lemma}\label{lem-BI+random}
Suppose that Assumptions \ref{asp-HF} and \ref{asp-HC} are satisfied.
Let $\{x_k\}$ be a sequence generated by Algorithm \ref{alg-Rand-QS}, and let $\mathcal{F}_k:=\{x_0,x_1,\dots,x_k\}$ for any $k\in \IN$.
Then it holds, for any $x^*\in X^*$ and any $k\in \IN$, that
\begin{equation}\label{eq-BI-r}
\mathbf{E}\left\{\|x_{k+1}-x^*\|^2\mid\mathcal{F}_k\right\}\le \|x_k-x^*\|^2-2v_k\frac{R_{p,m}}m (f(x_k)-f^*)^{\frac{1}{p}}+v_k^2.
\end{equation}
\end{lemma}
\begin{proof}
Fix $x^*\in X^*$ and $k\in \IN$. In view of Algorithm \ref{alg-Rand-QS} and by the nonexpansive property of projection operator, we have
\begin{equation}\label{eq-lem-BI-r1}
\|x_{k+1}-x^*\|^2\le \left\|x_k-v_k g_{k,\omega_k}-x^*\right\|^2 = \|x_k-x^*\|^2-2v_k \left\langle g_{k,\omega_k}, x_k-x^*\right\rangle+ v_k^2.
\end{equation}
Note by Assumption \ref{asp-HF} that $x^*\in X^*= \cap_{i=1}^m X_i^*$, and so $x^*\in X_{\omega_k}^*$. By Algorithm \ref{alg-Rand-QS}, one sees that $\omega_k\in I_k$ and thus $f_{\omega_k}(x_k)>f_{\omega_k}^*$. Then, Lemma \ref{lem-HC} is applicable (with $f_{\omega_k}$, $x_k$, $X_{\omega_k}^*$ in place of $h$, $x$, $X^*$) to concluding that
\[
\left\langle g_{k,\omega_k}, x_k-x^*\right\rangle \ge \left(\frac{f_{\omega_k}(x_k)-f_{\omega_k}(x^*)}{L_{\omega_k}}\right)^{\frac{1}{p}}\ge L_{\max}^{-\frac1p} \left(f_{\omega_k}(x_k)-f_{\omega_k}^*\right)^{\frac{1}{p}}.
\]
Then \eqref{eq-lem-BI-r1} is reduced to
\begin{equation*}\label{eq-lem-BI-r1a}
\|x_{k+1}-x^*\|^2\le \|x_k-x^*\|^2-2v_k L_{\max}^{-\frac1p} \left(f_{\omega_k}(x_k)-f_{\omega_k}^*\right)^{\frac{1}{p}}+ v_k^2.
\end{equation*}
Taking the conditional expectation with respect to $\mathcal{F}_k$, it follows that
\begin{equation}\label{eq-lem-BI-r2}
\mathbf{E}\left\{\|x_{k+1}-x^*\|^2\mid\mathcal{F}_k\right\}\le \|x_k-x^*\|^2
-2v_k L_{\max}^{-\frac1p}\mathbf{E}\left\{\left(f_{\omega_k}(x_k)-f_{\omega_k}^*\right)^{\frac{1}{p}}\mid\mathcal{F}_k\right\}+ v_k^2.
\end{equation}

Below, we provide an estimation of the term $\mathbf{E}\left\{\left(f_{\omega_k}(x_k)-f_{\omega_k}^*\right)^{\frac{1}{p}}\mid\mathcal{F}_k\right\}$.
Noting in Algorithm \ref{alg-Rand-QS} that $\omega_k$ is uniformly distributed on $I_k$, we have $P(\omega_k=i)=\frac{1}{|I_k|}$ for each $i\in I_k$, and then conclude by the elementary of probability theory that
\begin{equation}\label{eq-lem-BI-r3}
\begin{array}{lll}
\mathbf{E}\left\{\left(f_{\omega_k}(x_k)-f_{\omega_k}^*\right)^{\frac{1}{p}}\mid\mathcal{F}_k\right\}
&=\frac{1}{|I_k|}\sum_{i\in I_k}\left(f_i(x_k)-f_i^*\right)^{\frac1p}\\
&\ge \frac{1}{|I_k|} \min\left\{1,|I_k|^{1-\frac1p}\right\}\left(\sum_{i\in I_k}(f_i(x_k)-f_i^*)\right)^{\frac1p},
\end{array}
\end{equation}
where the inequality follows from Lemma \ref{lem-IneQ-HYMOR} (with $f_i(x_k)-f_i^*$ and $\frac1p$ in place of $a_i$ and $\gamma$). By the definition of $I_k$ (see Algorithm \ref{alg-Rand-QS}), it follows that $f_i(x_k)=f_i^*$ for each $i\notin I_k$, and so, by Assumption \ref{asp-HF}, one has
\begin{equation}\label{eq-lem-BI-r3a}
\sum_{i\in I_k}(f_i(x_k)-f_i^*)=\sum_{i=1}^m(f_i(x_k)-f_i(x^*))=f(x_k)-f^*.
\end{equation}
Note by $|I_k|\le m$ that
\begin{equation}\label{eq-lem-BI-r3b}
\frac{1}{|I_k|} \min\left\{1,|I_k|^{1-\frac1p}\right\}=\min\left\{|I_k|^{-1},|I_k|^{-\frac1p}\right\}\ge \frac1m\min\left\{1,m^{1-\frac1p}\right\}.
\end{equation}
Therefore, by \eqref{eq-lem-BI-r3a} and \eqref{eq-lem-BI-r3b}, \eqref{eq-lem-BI-r3} is reduced to that
\[
\mathbf{E}\left\{\left(f_{\omega_k}(x_k)-f_{\omega_k}^*\right)^{\frac{1}{p}}\mid\mathcal{F}_k\right\}
\ge \frac1m\min\left\{1,m^{1-\frac1p}\right\}\left(f(x_k)-f^*\right)^{\frac1p},
\]
which, together with \eqref{eq-lem-BI-r2} and \eqref{eq-para-r}, yields \eqref{eq-BI-r}.
The proof is complete.
\end{proof}

By virtue of Lemma \ref{lem-BI+random}, we explore the convergence properties (with probability 1) of the randomized incremental subgradient method when using different stepsize rules in Theorems \ref{thm-const-r}-\ref{thm-dyn-r}, respectively.

\begin{theorem}\label{thm-const-r}
Suppose Assumptions \ref{asp-HF} and \ref{asp-HC} are satisfied.
Let $\{x_k\}$ be a sequence generated by Algorithm \ref{alg-Rand-QS} with the constant stepsize rule {\rm (S1)}.
Then it holds, with probability 1, that
\begin{equation}\label{eq-const-r}
\liminf_{k\to \infty} f(x_k)\le f^*+\left(\frac{mv}{2R_{p,m}}\right)^p.
\end{equation}
\end{theorem}
\begin{proof}
Fix $\delta>0$, and define a set $X_\delta$ by
\[
X_\delta:=X \cap {\rm lev}_{< f^*+\left(\frac{mv}{2R_{p,m}}+\delta\right)^p} f.
\]
Let $y_\delta\in X$ be such that $f(y_\delta)=f^*+\delta^p$ (this $y_\delta$ is well-defined by the continuity of $f$). Hence $y_\delta\in X_\delta$ by construction.
We define a new process $\{\hat{x}_k\}$ by letting $\hat{x}_0:=x_0$ and
\[\hat{x}_{k+1}:=\left\{\begin{matrix}
   P_X \left(\hat{x}_k-v_k \hat{g}_{k,\hat{\omega}_k}\right),&{\mbox{if }\hat{x}_k\notin X_\delta,}\\ y_\delta, &{\rm otherwise},
\end{matrix}\right.\]
where $\hat{g}_{k,\hat{\omega}_k}\in \partial f_{\hat{\omega}_k}(\hat{x}_k) \cap S(0,1)$. Clearly, the process $\{\hat{x}_k\}$ is identical to $\{x_k\}$, except that $\hat{x}_k$ enters $X_\delta$ and then the process terminates with $\hat{x}_k = y_\delta\in X_\delta$.

Assume that $\hat{x}_k\notin X_\delta$ for any $k\in \IN$, and let $\mathcal{\hat{F}}_k:=\{\hat{x}_0,\hat{x}_1,\dots,\hat{x}_k\}$ for any $k\in \IN$.
It says that $f(\hat{x}_k)\ge f^*+\left(\frac{mv}{2R_{p,m}}+\delta\right)^p$, and then it follows from Lemma \ref{lem-BI+random} that the following relation holds for any $x^*\in X^*$ and any $k\in \IN$:
\[
\begin{array}{lll}
\mathbf{E}\left\{\|\hat{x}_{k+1}-x^*\|^2\mid\mathcal{\hat{F}}_k\right\}
&\le \|\hat{x}_k-x^*\|^2-2v\frac{R_{p,m}}m (f(\hat{x}_k)-f^*)^{\frac{1}{p}}+v^2\\
&\le \|\hat{x}_k-x^*\|^2-2v\delta\frac{R_{p,m}}m.
\end{array}
\]
Then, by Theorem \ref{thm-super}, we obtain that $\sum_{k=0}^\infty 2v\delta\frac{R_{p,m}}m<\infty$ with probability 1,
which is impossible.
Hence, $\hat{x}_k\notin X_\delta$ only occurs finitely many times, and so $\hat{x}_k\in X_\delta$ for large $k$. Consequently, in the original
process, it holds with probability 1 that
\[
\liminf_{k\to \infty} f(x_k)\le f^*+\left(\frac{mv}{2R_{p,m}}+\delta\right)^p.
\]
Since $\delta>0$ is arbitrary, \eqref{eq-const-r} is obtained by letting $\delta\to 0$, and the proof is complete.
\end{proof}

\begin{remark}\label{rem-CT-Ran}
{\rm
Theorem \ref{thm-const-r} depicts the convergence of Algorithm \ref{alg-Rand-QS} to the optimal value of problem \eqref{eq-SP} within a tolerance, expressed in terms of the stepsize, the number of component functions and parameters of H${\rm \ddot{o}}$lder conditions, when the constant stepsize rule is adopted.
It is observed by \eqref{eq-thm-cons} and \eqref{eq-const-r} that the randomized incremental subgradient method (Algorithm \ref{alg-Rand-QS}) admits a much less error bound than that of the incremental subgradient method (Algorithm \ref{alg-IncQS}) when adopting the same stepsize. Indeed,
\[
\frac{\left(\frac{mv}{2R_{p,m}}\right)^p}{\left(\frac{m^2v}{2C_{p,m}}\right)^p}=\frac{C_{p,m}^p}{R_{p,m}^pm^p}
=\frac{\min\left\{1,(2m)^{p-1}\right\}}{\min\left\{1,m^{p-1}\right\}m^p}
\le\frac{\max\left\{1,2^{p-1}\right\}}{m^p}\ll 1.
\]
}
\end{remark}

The proof of the following theorem uses the property of the diminishing stepsize rule (cf. \eqref{eq-dim+stepsize}) and a line of analysis similar to that of Theorem \ref{thm-const-r}. Hence we omit the details.

\begin{theorem}\label{thm-dimi-r}
Suppose Assumptions \ref{asp-HF} and \ref{asp-HC} are satisfied. Let $\{x_k\}$ be a sequence generated by Algorithm \ref{alg-Rand-QS} with the diminishing stepsize rule {\rm (S2)}. Then $\liminf_{k\to \infty} f(x_k)= f^*$ with probability 1.
\end{theorem}

\begin{theorem}\label{thm-dyn-r}
Suppose Assumptions \ref{asp-HF} and \ref{asp-HC} are satisfied.
Let $\{x_k\}$ be a sequence generated by Algorithm \ref{alg-Rand-QS} with the dynamic stepsize rule {\rm (S4)}.
Then $\{x_k\}$ converges to an optimal solution of problem \eqref{eq-SP} with probability 1.
\end{theorem}
\begin{proof}
By Lemma \ref{lem-BI+random} and \eqref{eq-dyn+stepsize-r}, it follows that, for any $x^*\in X^*$ and any $k\in \IN$,
\begin{equation*}\label{eq-thm-dyn-r}
\begin{array}{lll}
\mathbf{E}\left\{\|x_{k+1}-x^*\|^2\mid\mathcal{F}_k\right\}&\le \|x_k-x^*\|^2-2v_k\frac{R_{p,m}}m (f(x_k)-f^*)^{\frac{1}{p}}+v_k^2\\
&= \|x_k-x^*\|^2-\gamma_k(2-\gamma_k) \frac{R_{p,m}^2}{m^2}(f(x_k)-f^*)^{\frac{2}{p}}\\
&\le \|x_k-x^*\|^2-\underline{\gamma}(2-\overline{\gamma}) \frac{R_{p,m}^2}{m^2}(f(x_k)-f^*)^{\frac{2}{p}}.
\end{array}
\end{equation*}
Then it follows from Theorem \ref{thm-super} that $\{\|x_k-x^*\|\}$ is convergent and $\sum_{k=1}^\infty (f(x_k)-f^*)^{\frac{2}{p}}<\infty$ with probability 1; consequently, $\lim_{k \to \infty} f(x_k)=f^*$ with probability 1.

Let $(\Omega,\mathcal{F},P)$ be the probability space. Let $Z$ be a countable and dense subset of $X^*$, and let
\[
\Theta(z):=\left\{\omega: \{\|x_k(\omega)-z\|\} \mbox{ is convergent}\right\} \quad \mbox{ for any } z\in Z,
\]
and
\[
\Theta:=\bigcap_{z\in Z} \Theta(z).
\]
Recall that $\{\|x_k-x^*\|\}$ is convergent with probability 1, that is $P(\Theta(x^*))=1$, for any $x^*\in X^*$. Then it follows that $P\left(\Theta(z)^c\right)=0$ for any $z\in Z\subseteq X^*$. By the elements of probability theory, one checks that
\begin{equation}\label{eq-Ptheta}
P(\Theta)=1-P\left(\Theta^c\right)=1-P\left(\bigcup_{z\in Z} \Theta(z)^c\right)\ge 1-\sum_{z\in Z} P\left(\Theta(z)^c\right)=1.
\end{equation}
For any $\omega \in \Theta$ and any $z\in Z$, it says that $\{\|x_k(\omega)-z\|\}$ is convergent; hence $\{x_k(\omega)\}$ is bounded and must have a cluster point.
Define $\bar{x}:\Omega \to \R^n$ be such that
\[
\mbox{$\bar{x}(\omega)$ is a cluster point of $\{x_k(\omega)\}$ for any $\omega \in \Theta$}.
\]
Note again that $\lim_{k \to \infty} f(x_k)=f^*$ with probability 1. Without loss of generality, we can assume that $\lim_{k\to \infty}f\left(x_k(\omega)\right)= f^*$ for any $\omega \in \Theta$. Then it follows from the continuity of $f$ that
\begin{equation}\label{eq-omega-a}
\mbox{$\bar{x}(\omega) \in X^*$\quad for any $\omega \in \Theta$}.
\end{equation}
Fix $\epsilon>0$ and $\omega\in \Theta$. Since $\bar{x}(\omega) \in X^*$ and $Z\subseteq X^*$ is dense, there exists $z(\omega)\in Z$ such that
\begin{equation}\label{eq-thm-Dyn-1}
\|\bar{x}(\omega)-z(\omega)\|\le \frac{\epsilon}3.
\end{equation}
Let $\{x_{k_i}(\omega)\}$ be a subsequence of $\{x_k(\omega)\}$ such that $\lim_{i\to \infty}x_{k_i}(\omega)= \bar{x}(\omega)$. Hence we obtain by \eqref{eq-thm-Dyn-1} that $\lim_{i\to \infty} \|x_{k_i}(\omega)-z(\omega)\| \le \frac{\epsilon}3$.
By the definition of $\Theta$, one has that $\{\|x_k(\omega)-z(\omega)\|\}$ is convergent, and so $\lim_{k\to \infty} \|x_{k}(\omega)-z(\omega)\| \le \frac{\epsilon}3$. Then there exists $N\in \mathbb{N}$ such that
\begin{equation*}\label{eq-thm-Dyn-3}
\|x_k(\omega)-z(\omega)\|\le \frac{2\epsilon}3\quad \mbox{for any $k\ge N$}.
\end{equation*}
This, together with \eqref{eq-thm-Dyn-1}, yields
\begin{equation*}
\|x_k(\omega)-\bar{x}(\omega)\|\le \|x_k(\omega)-z(\omega)\| + \|\bar{x}(\omega)-z(\omega)\| \le \epsilon\quad \mbox{for any $k\ge N$}.
\end{equation*}
This shows that $\{x_k(\omega)\}$ converges to $\bar{x}(\omega)$ for any $\omega \in \Theta$. This, together with \eqref{eq-omega-a}, says that
\[
\Theta\subseteq \left\{\omega\in \Omega: \{x_k(\omega)\} \mbox{ converges to } \bar{x}(\omega)\right\}\cap
\left\{\omega\in \Omega: \bar{x}(\omega) \in X^*\right\}.
\]
Noting by \eqref{eq-Ptheta} that $P(\Theta)=1$, we conclude
\[
P\left(\left\{\omega\in \Omega: \{x_k(\omega)\} \mbox{ converges to } \bar{x}(\omega),\, \bar{x}(\omega) \in X^*\right\}\right)\ge P(\Theta)=1.
\]
The proof is complete.
\end{proof}

\section{Application to sum of ratios problems}
This section aims to present an important application of sum-minimization problem \eqref{eq-SP} of a number of quasi-convex component functions.
Typically, fractional programming, optimizing a certain indicator (e.g. efficiency) characterized by a ratio of technical terms, is widely applied in various areas; see \cite{Avriel10,CrouzeixMM98,Stancu97} and references therein. In particular, the sum of ratios problem (SOR) \cite{SchaibleSOR03} is a typical fractional programming and has a variety of important applications in economics and management science, which is formulated as
\begin{equation}\label{eq-SOR}
  \begin{array}{ll}
  \text{max}& \sum_{i=1}^m R_i(x):=\frac{p_i(x)}{c_i(x)}\\
     \text{s.t.}  & x\in X,
    \end{array}
\end{equation}
where $p_i:\R^n\to \R$ is nonnegative and concave, $c_i:\R^n\to \R$ is positive and convex for each $i=1,\dots,m$. It is difficult to globally solve the SOR \eqref{eq-SOR}, especially for large-scale problems.

Exploiting the additivity structure of problem \eqref{eq-SOR}, we propose a new approach to find a global optimal solution of the SOR. By \cite[Theorems 2.3.3 and 2.5.1]{Stancu97}, we have that the ratio $R_i$ is quasi-concave for each $i=1,\dots,m$, and so \eqref{eq-SOR} is a sum-maximization problem of a number of quasi-concave functions. This shows that the SOR falls in the framework \eqref{eq-SP}. Moreover, let $r_i$ denote the maximal ratio of $R_i$ over $X$, and define $h_i(\cdot):=r_i-R_i(\cdot)$. The SOR \eqref{eq-SOR} can also be approached by solving the resulting quasi-convex feasibility problem \eqref{eq-FP}.
In \cite{CensorSegal06}, Censor and Segal proposed a subgradient projection method to solve the quasi-convex feasibility problem \eqref{eq-FP} by using the dynamic stepsize rule.
In the numerical study, we apply the incremental subgradient methods and the subgradient projection method to solve the SOR \eqref{eq-SOR} and its reformulated quasi-convex feasibility problem \eqref{eq-FP}, respectively, and the abbreviations of these methods are listed in Table \ref{tab-algorithms}.

\begin{table}[h]\footnotesize
  \centering
  \caption{List of the algorithms compared in the numerical study.}
\begin{tabular}{|l|l|}
\hline
Abbreviations  &  \multicolumn{1}{|c|}{Algorithms}    \\ \hline
SGPM  &  \textbf{S}ub\textbf{G}radient \textbf{P}rojection \textbf{M}ethod in \cite{CensorSegal06}, which is to solve \eqref{eq-FP}.      \\ \hline
IncSGM  &  \textbf{Inc}remental \textbf{S}ub\textbf{G}radient \textbf{M}ethod (Algorithm \ref{alg-IncQS}) for solving \eqref{eq-SOR}.     \\ \hline
RandSGM  & \textbf{Rand}omized incremental \textbf{S}ub\textbf{G}radient \textbf{M}ethod (Algorithm \ref{alg-Rand-QS}) for solving \eqref{eq-SOR}. \\ \hline
\end{tabular}
\label{tab-algorithms}
\end{table}

In the numerical study, we consider the multiple Cobb-Douglas productions efficiency problem (in short, MCDPE) \cite{BradleyFrey74}, which is an application of the SOR.
Formally, consider a set of $m$ productions with $s$ projects and $n$ factors. Let $x:=(x_j)^T\in \R^n$ denote the amounts of $n$ factors. The profit function of production $i$ can be expressed as the Cobb-Douglas production function
\[
p_i(x):=a_{i,0}\prod_{j=1}^n x_j^{a_{i,j}},
\]
where $a_{i,j}\ge 0$ for $j=0,\dots,n$ and $\sum_{j=1}^n a_{i,j}=1$. The cost function of production $i$ is formulated as a linear function
\[
c_i(x):=\sum_{j=1}^n c_{i,j}x_j+c_{i,0},
\]
where $c_{i,j}\ge 0$ for $j=0,\dots,n$.
Due to the daily profit or operating cost constraints, the amounts of investment for factors should fall in the constraint set
\[
X:=\{x\in \R^n_+:\sum_{j=1}^n b_{tj}x_j\ge p_t,\quad t=1,\dots,s\}.
\]
Then the MCDPE is modeled as the SOR \eqref{eq-SOR}.
In the numerical experiments, the parameters of MCDPE are randomly chosen from different intervals:
\[
a_{i,0}\in [0,10],\quad a_{i,j},b_{tj},c_{i,0},c_{i,j}\in [0,1],\quad {\rm and}\quad p_t\in [0,n/2].
\]
The diminishing stepsize rule is chosen as
\[
v_k=v/(1+0.1k),
\]
where $v$ is always chosen between $[2,5]$, while the constant stepsize is selected between $[1,2]$. 
All numerical experiments are implemented in MATLAB R2014a and executed on a personal laptop (Intel Core i5, 3.20 GHz, 8.00 GB of RAM).

We first compare the performances (in both the obtained objective value and the CPU time) of the SGPM, IncSGM and RandSGM for different dimensions. The computation results are displayed in Table \ref{tab-performance}.
In this table, the columns of Projects, Factors and Productions represent the numbers of projects ($s$), factors ($n$) and productions ($m$) of MCDPE, and the columns of $f_{opt}$ and CPUtime denote the obtained optimal value and the CPU time (seconds) cost to reach $f_{opt}$ by each algorithm, respectively.
It is observed from Table \ref{tab-performance} that the IncSGM and RandSGM outperform the SGPM in the sense that they achieve a larger production efficiency in a shorter time than the SGPM for different dimensional MCDPEs.

\begin{table}[h]\footnotesize
  \centering
  \caption{Computation results for maximizing MCDPE.}
\begin{tabular}{|ccc|cc|cc|cc|}
\hline
\multicolumn{3}{|c|}{Circumstance of problem } &\multicolumn{2}{|c|}{SGPM} &\multicolumn{2}{|c|}{IncSGM} &\multicolumn{2}{|c|}{RandSGM} \\ \hline
Projects  & Factors & Productions &  $f_{opt}$ & CPUtime     &  $f_{opt}$ & CPUtime     &  $f_{opt}$ & CPUtime               \\ \hline
50        &50　　   &10           &  23.31    & 0.51     &  23.46    & 0.17     &  23.48    & 0.18               \\ \hline
50        &50　　   &100          &  210.22    & 3.38     &  211.86    & 2.41     &  211.84    & 1.74               \\ \hline
100       &100 　   &10           &  11.73    & 0.41     &  11.77    & 0.26     &  11.81    & 0.23               \\ \hline
100       &100 　   &100          &  104.20    & 2.62     &  106.52    & 1.40     &  106.49    & 1.03               \\ \hline
500       &500 　　 &10           &  2.21    & 1.45     &  2.31    & 0.54     &  2.34    & 0.38               \\ \hline
500       &500 　　 &100          &  21.01    & 9.61     &  21.25    & 5.93     &  21.24    & 4.28               \\ \hline
1000      &1000  　 &10           &  1.15    & 3.23     &  1.19    & 1.69     &  1.21    & 1.47               \\ \hline
1000      &1000  　 &100          &  10.56    & 19.64     &  10.62    & 12.48     &  10.60    & 10.41               \\ \hline
\end{tabular}
\label{tab-performance}
\end{table}

The second experiment is to compare the convergence behavior of the SGPM, IncSGM and RandSGM by using the constant and diminishing stepsize rules, where the problem size is fixed to be $(m,n,s)=(10, 100, 100)$. We summary the averaged performance of the compared algorithms in 500 random trials.
Figure \ref{fig-ave-convergence} plots the mean of the estimated Cobb-Douglas production efficiencies along the number of the iterations in these 500 trials, from which we observe that the IncSGM converges faster (in terms of the number of iterations) to an (approximate) optimal value that the RandSGM and the SGPM. Furthermore, Figure \ref{fig-ave-convergence}(a) illustrates that the RandSGM obtains a better estimation than the IncSGM when the constant stepsize rule is adopted, which is consistent with Remark \ref{rem-CT-Ran}. Figure \ref{fig-ave-convergence}(b) demonstrates that both IncSGM and RandSGM converge to an optimal value when the diminishing stepsize rule is employed, which is consistent with Theorems \ref{thm-dimi} and \ref{thm-dimi-r}. It is also shown that both IncSGM and RandSGM approach a better solution that the SGPM.
Figure \ref{fig-ave-time} plots the error bars of the CPU times in 500 trials when varying the number of component functions from 10 to 200. It is revealed that the RandSGM is faster (in terms of CPU time) than the IncSGM, which is faster than the SGPM. This indicates the potential applicability of the RandSGM to the large-scale SOR. Figure \ref{fig-500trials} plots the obtained maximal production efficiencies in each of these 500 trials. It is observed that the IncSGM and RandSGM outperform the SGPM consistently.

\begin{figure}[h]
\centering
\mbox{ \subfigure[The constant stepsize rule.]{\includegraphics[width=6.5cm]{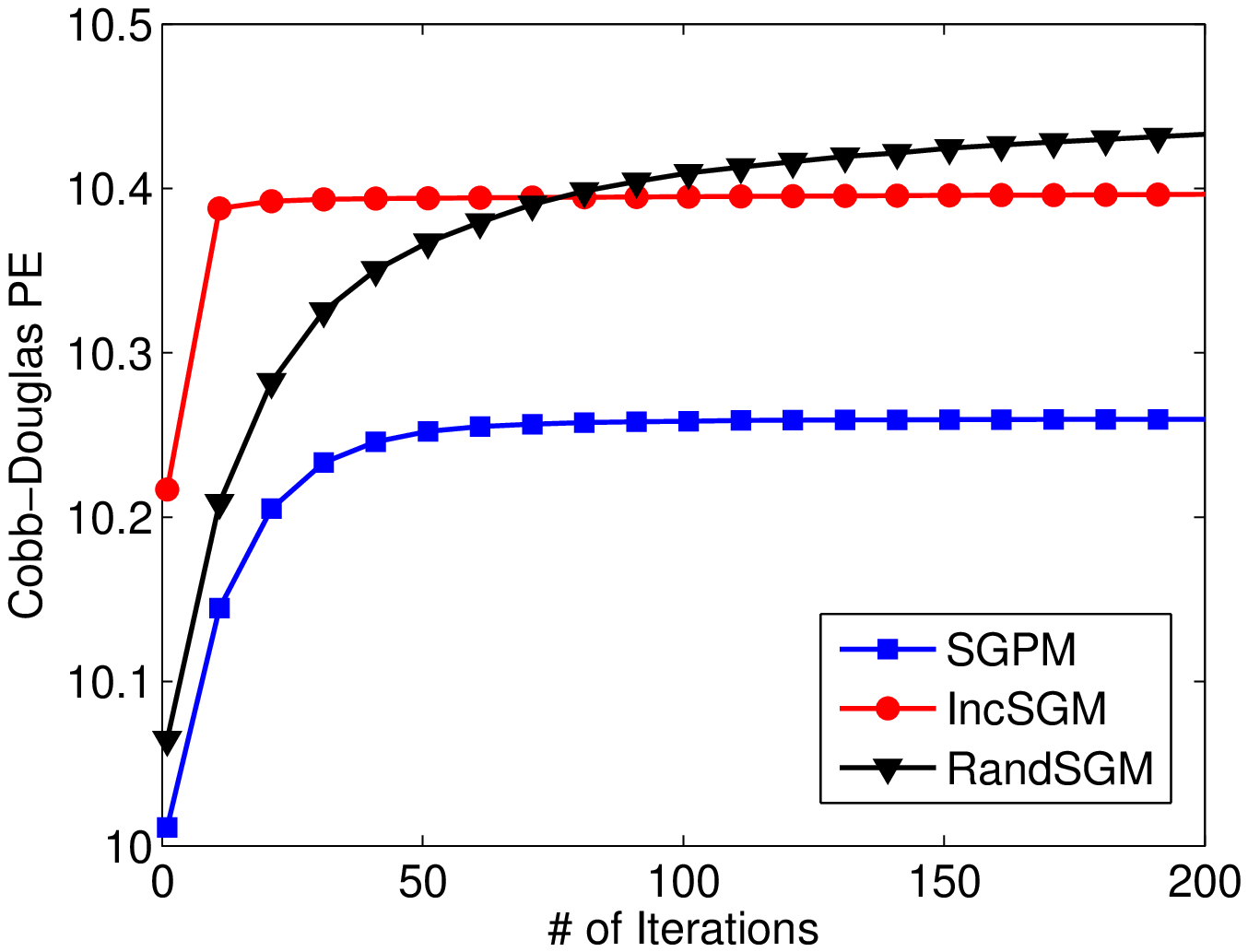}} \quad
\subfigure[The diminishing stepsize rule.]{\includegraphics[width=6.5cm]{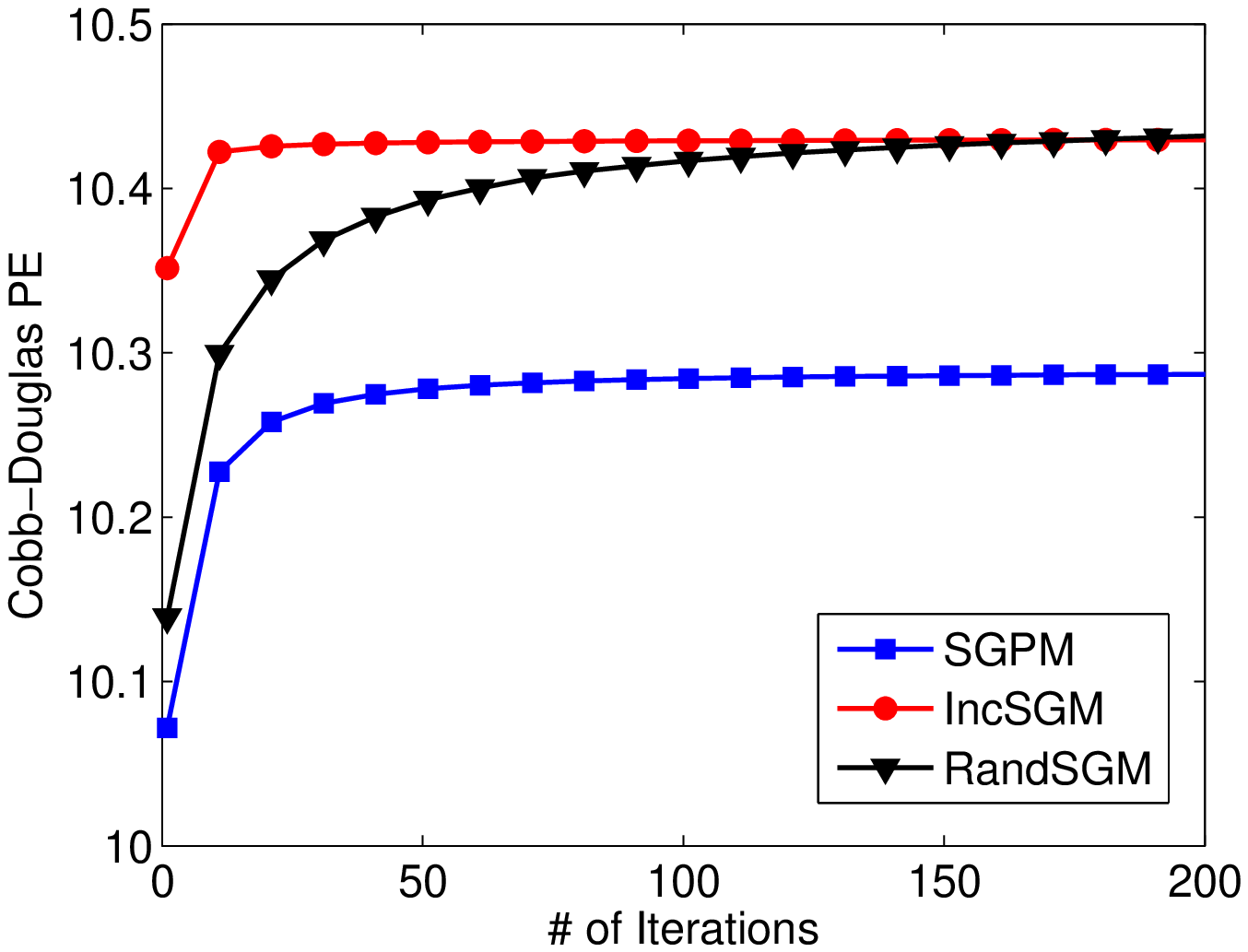}}
}
\caption{The averaged convergence behavior of SGPM, IncSGM and RandSGM in 500 random MCDPEs.}
\label{fig-ave-convergence}
\end{figure}

\begin{figure}[h]
\centering
\mbox{ \subfigure[The constant stepsize rule.]{\includegraphics[width=6.5cm]{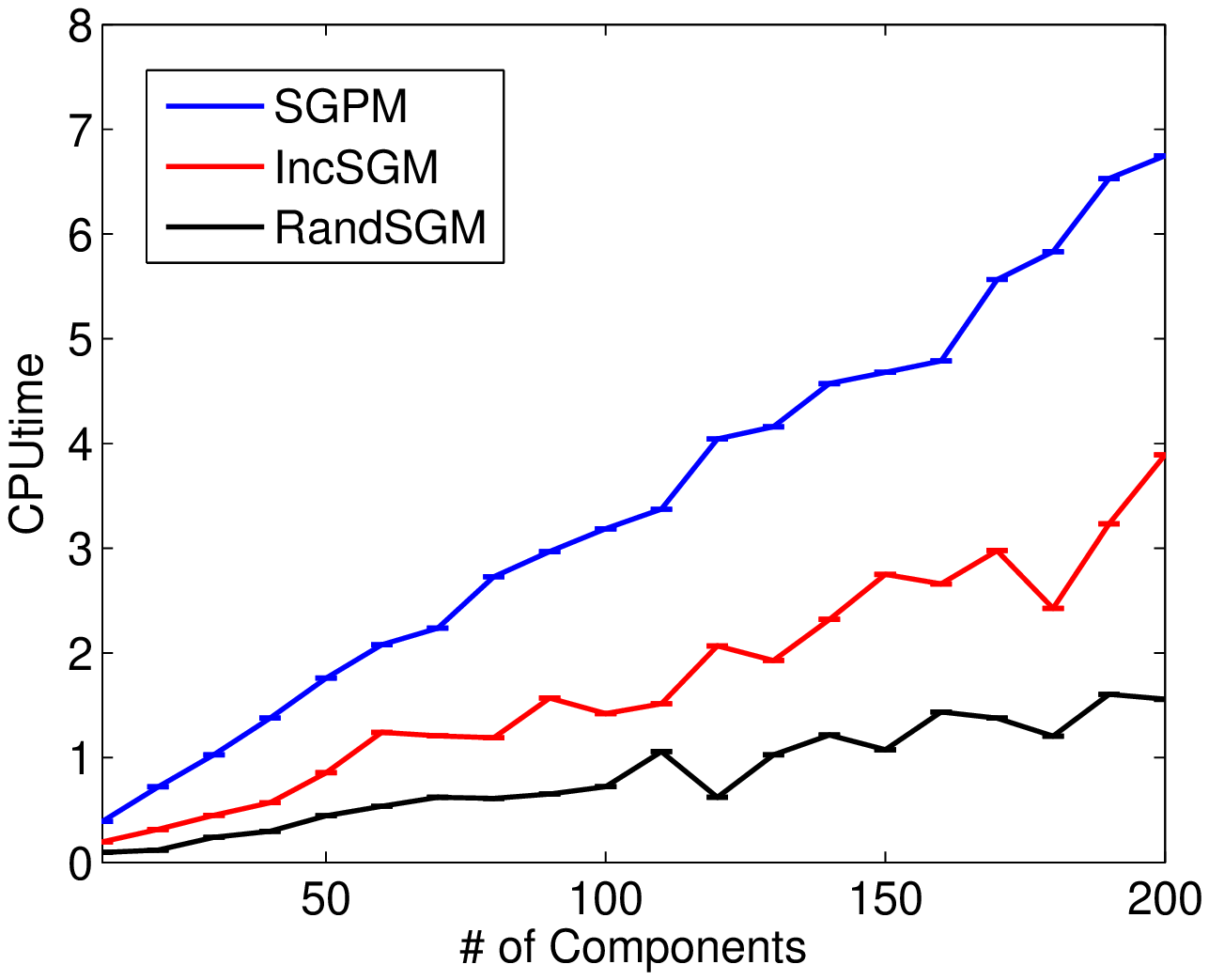}} \quad
\subfigure[The diminishing stepsize rule.]{\includegraphics[width=6.5cm]{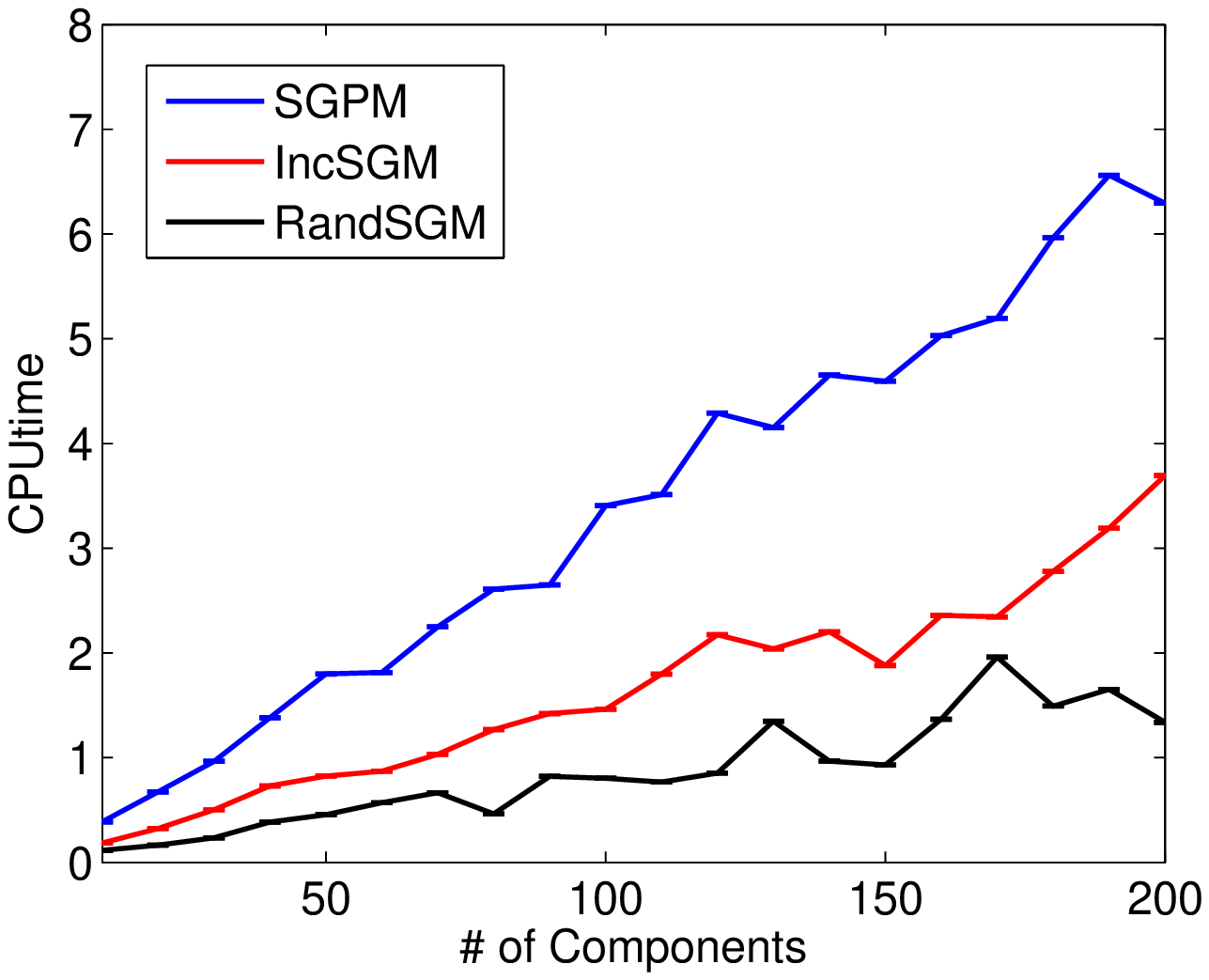}}
}
\caption{Variation of CPU time when varying the number of component functions.}
\label{fig-ave-time}
\end{figure}

\begin{figure}[h]
\centering
\mbox{ \subfigure[The constant stepsize rule.]{\includegraphics[width=6.5cm]{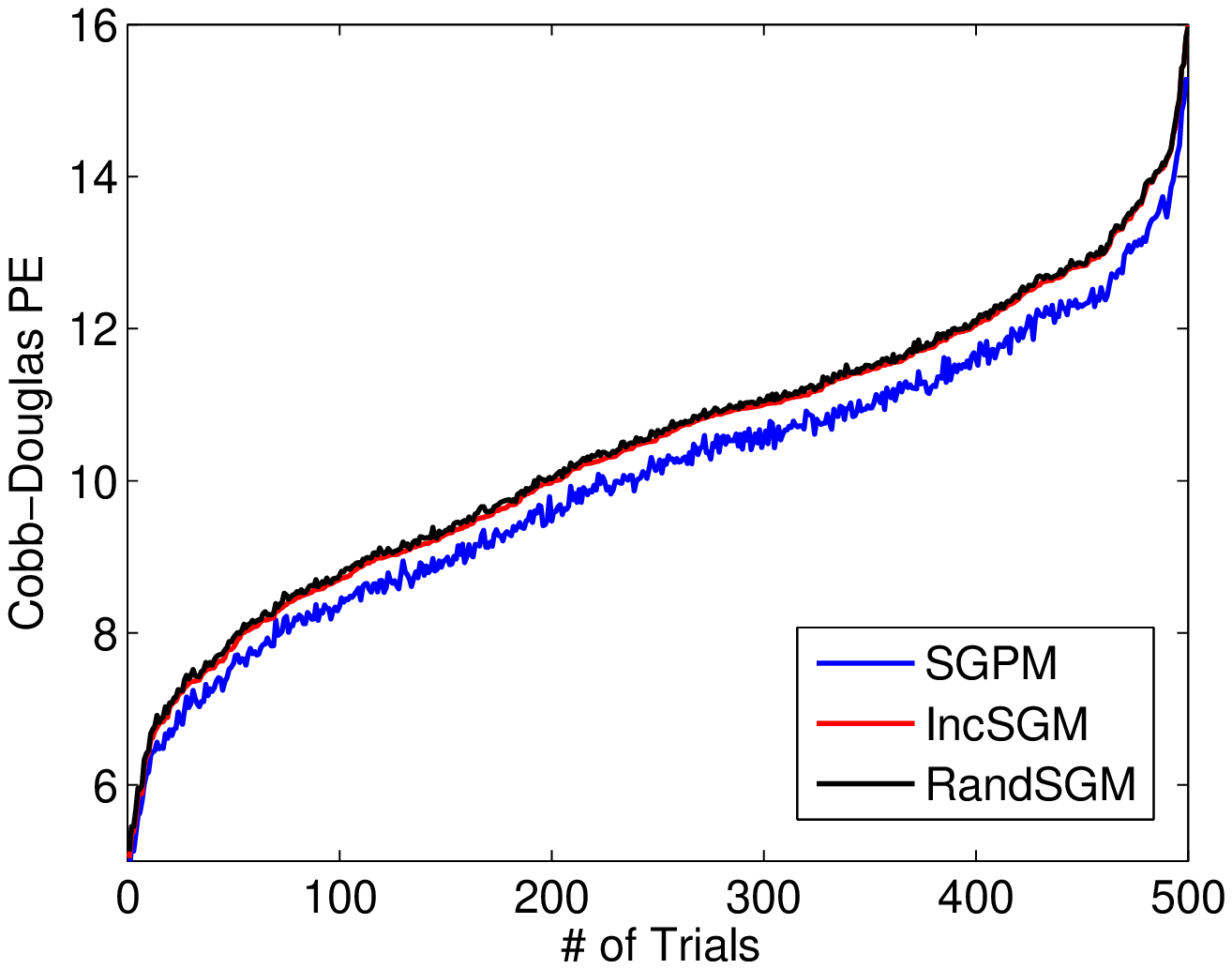}} \quad
\subfigure[The diminishing stepsize rule.]{\includegraphics[width=6.5cm]{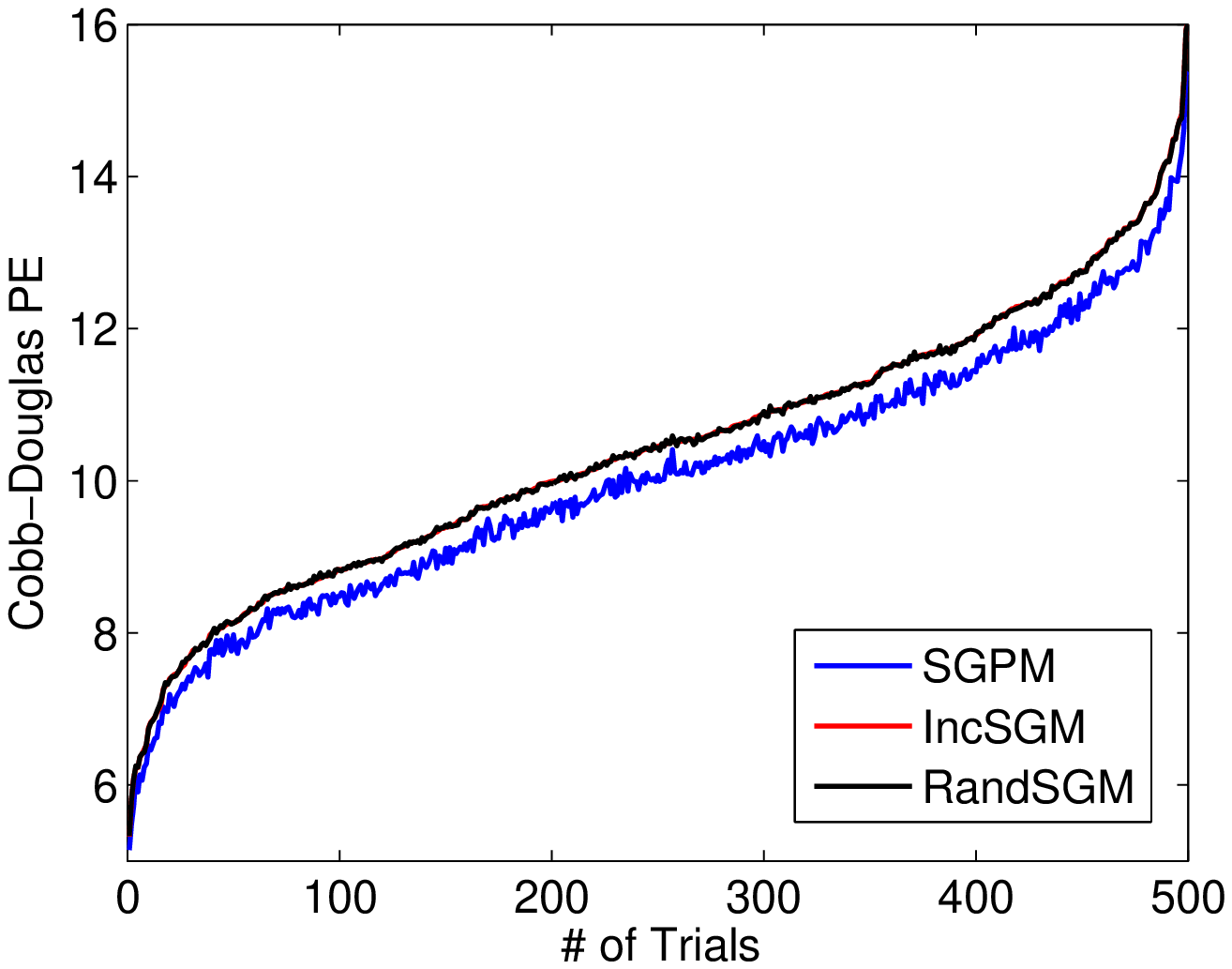}}
}
\caption{Overall of obtained maximal values in 500 random MCDPEs.}
\label{fig-500trials}
\end{figure}


\begin{figure}[h]
\centering
  \includegraphics[width=7cm]{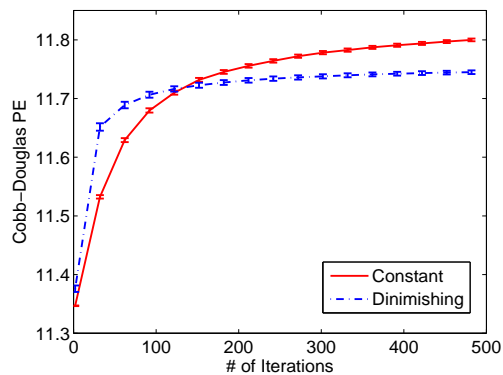}\\
  \caption{The error bars of RandSGM in 500 simulations.}
  \label{fig-ErrorBar}
\end{figure}

Finally, we conduct 500 simulations to show the stability of RandSGM, which start from the same initial point, adopt the same stepsizes (constant: $v_k\equiv 1.5$ or diminishing: $v_k=3/(1+0.1k)$) and solve a same MCDPE, but follow different stochastic processes. Figure \ref{fig-ErrorBar} plots the error bars of the estimated Cobb-Douglas production efficiencies in these 500 simulations. It is shown that the RandSGM is highly stable and converges to an optimal value with probability 1.

%


\end{document}